\documentclass[a4paper,12pt,leqno,twoside]{article}
\usepackage{amsmath, amsthm, amsfonts, amssymb, graphicx, color, cancel, ulem}
\usepackage{esint}
\numberwithin{equation}{section}
     \newtheorem{thm}{Theorem}[section]
     \newtheorem{cor}[thm]{Corollary}
     \newtheorem{prop}[thm]{Proposition}
     \newtheorem{lem}[thm]{Lemma}
\theoremstyle{definition}
      \newtheorem{defn}{Definition}[section]
     
\theoremstyle{remark}
     \newtheorem{rem}{Remark}[section]

\newcommand{\R}{\mathbb{R}}

\newcommand{\Z}{\mathbb{Z}}

\newcommand{\cC}{\mathcal{C}}

\newcommand{\cL}{\mathcal{L}}

\newcommand{\cS}{\mathcal{S}}

\newcommand{\BMO}{\mathrm{BMO}}
\newcommand{\Lip}{\mathrm{Lip}}

\newcommand{\supp}{\operatorname{supp}}

\newcommand{\loc}{\mathrm{loc}}
\newcommand{\comp}{\mathrm{comp}}

\newcommand{\ls}{\lesssim}
\newcommand{\gs}{\gtrsim}

\newcommand{\vp}{\varphi}

\newcommand{\tvp}{\tilde\varphi}

\newcommand{\cPhi}{\widetilde\Phi}

\newcommand{\cGdec}{\mathcal{G}^{\rm dec}}
\newcommand{\cGinc}{\mathcal{G}^{\rm inc}}

\newcommand{\Cic}{C^{\infty}_{\comp}}

\newcommand{\LP}{L^{\Phi}}
\newcommand{\LPs}{L^{\Psi}}

\newcommand{\LcP}{L^{\widetilde\Phi}}

\newcommand{\LPp}{L^{(\Phi,\vp)}}
\newcommand{\cLPp}{\mathcal{L}^{(\Phi,\vp)}}

\newcommand{\LPsp}{L^{(\Psi,\vp)}}

\newcommand{\LPsps}{L^{(\Psi,\psi)}}
\newcommand{\cLPsps}{\mathcal{L}^{(\Psi,\psi)}}

\newcommand{\LTp}{L^{(\Theta,\vp)}}

\newcommand{\Ir}{I_{\rho}}
\newcommand{\Ia}{I_{\alpha}}

\newcommand{\Ma}{M_{\alpha}}
\newcommand{\Mr}{M_{\rho}}
\newcommand{\Ms}{M^{\sharp}}
\newcommand{\Md}{M^{\mathrm{dy}}}

\newcommand{\Msd}{M^{\sharp,\mathrm{dy}}}
\newcommand{\cQd}{\mathcal{Q}^{\mathrm{dy}}}

\newcommand{\biP}{{\it{\overline\varPhi}}}
\newcommand{\iPy}{{\it{\Phi_Y}}}

\newcommand{\biPy}{{\it{\overline\Phi_Y}}}

\newcommand{\dtwo}{\Delta_2}
\newcommand{\ntwo}{\nabla_2}
\newcommand{\bdtwo}{\overline\Delta_2}
\newcommand{\bntwo}{\overline\nabla_2}

\newcommand{\dlim}{\displaystyle\lim}

\newcommand{\sgn}{\mathrm{sgn}}

\newcommand{\msckw}{%
\footnotetext{\hspace{-0.35cm} 2010 {\it Mathematics Subject Classification}. 
42B35, 46E30, 42B20, 42B25.
\endgraf{\it Key words and phrases.} 
Orlicz-Morrey space, Campanato space, singular integral, fractional integral, commutator. \par
}
}

\begin{document}

\baselineskip=18pt

\title{%
Commutators of integral operators with functions in Campanato spaces
on Orlicz-Morrey spaces
\msckw
}
\author{Minglei Shi, Ryutaro Arai and Eiichi Nakai}
\date{}

\maketitle

\begin{abstract}
We consider the commutators $[b,T]$ and $[b,I_{\rho}]$ 
on Orlicz-Morrey spaces,
where $T$ is a Calder\'on-Zygmund operator, 
$I_{\rho}$ is a generalized fractional integral operator
and $b$ is a function in generalized Campanato spaces.
We give a necessary and sufficient condition for the boundedness
of the commutators on Orlicz-Morrey spaces.
To do this 
we prove the boundedness of generalized fractional maximal operators
on Orlicz-Morrey spaces.
Moreover,
we introduce Orlicz-Campanato spaces
and establish their relations to Orlicz-Morrey spaces.
\end{abstract}

\section{Introduction}\label{sec:intro}

Let $\R^n$ be the $n$-dimensional Euclidean space.
Let $b\in\BMO(\R^n)$
and $T$ be a Calder\'on-Zygmund singular integral operator.
In 1976
Coifman, Rochberg and Weiss~\cite{Coif-Roch-Weiss1976}
proved that the commutator $[b,T]=bT-Tb$ 
is bounded on $L^p(\R^n)$ ($1<p<\infty$), that is,
\begin{equation*}
 \|[b,T]f\|_{L^p}=\|bTf-T(bf)\|_{L^p}\le C\|b\|_{\BMO}\|f\|_{L^p},
\end{equation*}
where $C$ is a positive constant independent of $b$ and $f$.
For the fractional integral operator $\Ia$, 
Chanillo~\cite{Chanillo1982} 
proved the boundedness of $[b,\Ia]$ in 1982.
Coifman, Rochberg and Weiss~\cite{Coif-Roch-Weiss1976} and Chanillo~\cite{Chanillo1982} 
also gave the necessary conditions for the boundedness.
These results were extended to 
Orlicz spaces by Janson~\cite{Janson1978} (1978)
and to Morrey spaces by Di~Fazio and Ragusa~\cite{DiFazio-Ragusa1991} (1991).
For other extensions and generalization, see 
\cite{Arai-Nakai2018RMC,Fu-Yang-Yuan2012,Fu-Yang-Yuan2014,Guliyev-Deringoz-Hasanov2017,Iida2016,Komori-Mizuhara2003,Mizuhara1999,Nakamura-Sawano2017Coll,Shirai2006SCMJ,Shirai2006Hokkaido}, etc.

In this paper we investigate the commutators $[b,T]$ and $[b,\Ir]$
on Orlicz-Morrey spaces, 
where $T$ is a Calder\'on-Zygmund operator, 
$\Ir$ is a generalized fractional integral operator 
and $b$ is a function in generalized Campanato spaces.
The Orlicz-Morrey spaces unify Orlicz and Morrey spaces,
and the Campanato spaces unify $\BMO$ and Lipschitz spaces.
Therefore, our results contain many previous results as corollaries.
The boundedness of $T$ and $\Ir$ on the Orlicz-Morrey spaces
are known by \cite{Nakai2008KIT} and \cite{Nakai2008Studia}, respectively.
To prove the boundedness of $[b,T]$ and $[b,\Ir]$,
we need the generalized fractional maximal operator $\Mr$
and the sharp maximal operator $\Ms$.
We show the boundedness of $\Mr$ on Orlicz-Morrey spaces
under weaker conditions than $\Ir$.
Moreover,
we introduce Orlicz-Campanato spaces
and establish their relations to Orlicz-Morrey spaces.

First we recall the Orlicz-Morrey space.
We denote by $B(a,r)$ 
the open ball centered at $a\in\R^n$ and of radius $r$.
For a function $f\in L^1_{\loc}(\R^n)$ and a ball $B$, 
let 
\begin{equation*}\label{mean}
  f_{B}=\fint_{B} f=\fint_{B} f(y)\,dy=\frac1{|B|}\int_{B} f(y)\,dy,
\end{equation*}
where $|B|$ is the Lebesgue measure of the ball $B$.

\begin{defn}[Orlicz-Morrey space]\label{defn:OrliczMorrey}
For a Young function $\Phi:[0,\infty]\to[0,\infty]$,
a function $\vp:(0,\infty)\to(0,\infty)$
and a ball $B=B(a,r)$, 
let
\begin{equation}\label{PpB norm}
 \|f\|_{\Phi,\vp,B}
 = 
 \inf\left\{ \lambda>0: 
  \frac{1}{\vp(r)}
  \fint_B \!\Phi\!\left(\frac{|f(x)|}{\lambda}\right)\! dx \le 1
           \right\}.
\end{equation}
Let $\LPp(\R^n)$ be the set of all functions $f$ such that the 
following functional is finite: 
\begin{equation}\label{OM norm}
 \|f\|_{\LPp} 
 =
 \sup_{B}  \|f\|_{\Phi,\vp,B},
\end{equation}
where the supremum is taken over all balls $B$ in $\R^n$.
(For the definition of the Young function, see the next section.)
\end{defn}

Then $\|f\|_{\LPp}$ is a norm and thereby $\LPp(\R^n)$ is a Banach space.
If $\vp(r)=1/r^{n}$, 
then $\LPp(\R^n)$ coincides with the Orlicz space $\LP(\R^n)$
equipped with the norm 
\begin{equation*}
 \|f\|_{\LP} 
 =
 \inf\left\{\lambda>0:
  \int_{\R^n} \Phi\left(\frac{|f(x)|}{\lambda}\right) dx
      \le 1
     \right\}.
\end{equation*}
If $\Phi(t)=t^p$, $1\le p<\infty$, 
then $\LPp(\R^n)$ coincides with the generalized Morrey space $L^{(p,\vp)}(\R^n)$
equipped with the norm 
\begin{equation*}
 \|f\|_{L^{(p,\vp)}} 
 =
 \sup_{B=B(a,r)} \left(\frac1{\vp(r)}\fint_B |f(x)|^p\,dx\right)^{1/p}.
\end{equation*}

The Orlicz-Morrey space $\LPp(\R^n)$ was first studied in \cite{Nakai2004KIT}.
For other kinds of Orlicz-Morrey spaces, see 
\cite{Deringoz-Guliyev-Nakai-Sawano-Shi2019Posi,Deringoz-Guliyev-Samko2014,Guliyev-Hasanov-Sawano-Noi2016,Sawano-Sugano-Tanaka2012}, etc.

Secondly, we recall the definition of the generalized Campanato space. 

\begin{defn}\label{defn:gC}
For $p\in[1,\infty)$ and a function $\psi:(0,\infty)\to(0,\infty)$, 
let
$\cL_{p,\psi}(\R^n)$ be the set of all functions $f$ such that the 
following functional is finite: 
\begin{equation*}
 \|f\|_{\cL_{p,\psi}}
  =\sup_{B=B(a,r)}\frac1{\psi(r)}\left(\fint_{B}|f(y)-f_{B}|^p\,dy\right)^{1/p},
\end{equation*}
where
the supremum is taken over all balls $B(a,r)$ in $\R^n$.
\end{defn}
Then $\|f\|_{\cL_{p,\psi}(\R^n)}$ is a norm modulo constant functions 
and thereby $\cL_{p,\psi}(\R^n)$ is a Banach space.
If $p=1$ and $\psi\equiv1$, then $\cL_{p,\psi}(\R^n)=\BMO(\R^n)$.
If $p=1$ and $\psi(r)=r^{\alpha}$ ($0<\alpha\le1$), 
then $\cL_{p,\psi}(\R^n)$ coincides with $\Lip_{\alpha}(\R^n)$.
If $\psi$ is almost increasing (see \eqref{almost}), 
then $\cL_{p,\psi}(\R^n)=\cL_{1,\psi}(\R^n)$.

Thirdly, we recall the generalized fractional integral operator $\Ir$.
For a function $\rho:(0,\infty)\to(0,\infty)$,
the operator $\Ir$ is defined by
\begin{equation}\label{Ir}
 \Ir f(x)=\int_{\R^n}\frac{\rho(|x-y|)}{|x-y|^n}f(y)\,dy,
 \quad x\in\R^n,
\end{equation}
where we always assume that
\begin{equation}\label{int rho}
 \int_0^1\frac{\rho(t)}{t}\,dt<\infty.
\end{equation}
If $\rho(r)=r^{\alpha}$, $0<\alpha<n$, 
then $\Ir$ is the usual fractional integral operator $\Ia$.
The condition \eqref{int rho} is needed for the integral in \eqref{Ir} 
to converge for bounded functions $f$ with compact support.
In this paper we also assume that 
there exist positive constants $C$, $K_1$ and $K_2$ with $K_1<K_2$ such that, for all $r>0$,
\begin{equation}\label{sup rho}
 \sup_{r\le t\le 2r}\rho(t)
 \le
 C\int_{K_1r}^{K_2r}\frac{\rho(t)}{t}\,dt.
\end{equation}
The condition \eqref{sup rho} was considered in \cite{Perez1994}.
The operator $\Ir$ was introduced in \cite{Nakai2001Taiwan} 
whose partial results were announced in 
\cite{Nakai2000ISAAC}.
See also \cite{Nakai2001SCMJ,Nakai2002Lund,Nakai2004KIT,Nakai2008Studia,Nakai-Sumitomo2001SCMJ}.

In this paper we prove the boundedness of $[b,T]$ and $[b,\Ir]$ such that
\begin{align*}
 \|[b,T]f\|_{\LPsp}
 &\le 
 C\|b\|_{\cL_{1,\psi}}\|f\|_{\LPp},
\\
 \|[b,\Ir]f\|_{\LPsp}
 &\le 
 C\|b\|_{\cL_{1,\psi}}\|f\|_{\LPp},
\end{align*}
under suitable assumptions.
We also prove that, 
if $[b,T]$ or $[b,\Ir]$ is bounded from $\LPp(\R^n)$ to $\LPsp(\R^n)$,
then $b$ is  in $\cL_{1,\psi}(\R^n)$
and the operator norms dominated by $\|b\|_{\cL_{1,\psi}}$.
To do this we first recall generalized Young functions and
define Orlicz-Morrey spaces with generalized Young functions
in Section~\ref{sec:Orlicz}.
Then we state the main results in Section~\ref{sec:main}.
We need generalized Young functions to show the boundedness of the commutators.
In Section~\ref{sec:prop}
we give basic properties on generalized Young functions and Orlicz-Morrey spaces.

To prove the main results
we show the boundedness of 
the generalized fractional maximal operator $\Mr$ in Section~\ref{sec:Mr}.
For a function $\rho:(0,\infty)\to(0,\infty)$, let
\begin{equation}\label{Mr}
 \Mr f(x)=\sup_{B(a,r)\ni x}\rho(r)\fint_{B(a,r)}|f(y)|\,dy,
 \quad x\in\R^n,
\end{equation}
where the supremum is taken over all balls $B(a,r)$ containing $x$.
We do not assume the condition \eqref{int rho} or \eqref{sup rho} on the definition of $\Mr$.
The operator $\Mr$ was studied in \cite{Sawano-Sugano-Tanaka2011} on generalized Morrey spaces.
If $\rho(r)=|B(0,r)|^{\alpha/n}, 0<\alpha<n$, 
then $M_{\rho}$ is the usual fractional maximal operator $\Ma$.
If $\rho\equiv1$, then $\Mr$ is 
the Hardy-Littlewood maximal operator $M$.
It is known that 
the usual fractional maximal operator $\Ma$ 
is dominated pointwise by 
the fractional integral operator $\Ia$,
that is,
$\Ma f(x)\le C\Ia|f|(x)$ for all $x\in\R^n$.
Then the boundedness of $\Ma$ follows from one of $\Ia$.
However, 
we need a better estimate on $\Mr$ than $\Ir$
to prove the boundedness of the commutators  $[b,T]$ and $[b,\Ir]$.

We also need the sharp maximal operator $\Ms$ defined by 
\begin{equation}\label{Ms}
 \Ms f(x)=\sup_{B\ni x}\fint_{B}|f(y)-f_B|\,dy,
 \quad x\in\R^n,
\end{equation}
where the supremum is taken over all balls $B$ containing $x$.
In section~\ref{sec:OMOC}
we show that,
if $f_{B(0,r)}\to0$ as $r\to\infty$,
then
\begin{equation}\label{}
 \|f\|_{\LPp}
 \le
 C\|\Ms f\|_{\LPp},
\end{equation}
by using the relation between Orlicz-Campanato and Orlicz-Morrey spaces.
In Section~\ref{sec:well} 
we show the well definedness of the commutators $[b,T]$ and $[b,\Ir]$
for functions in Orlicz-Morrey spaces.
Since neither $C^{\infty}_{\comp}(\R^n)$ nor $L^{\infty}_{\comp}(\R^n)$ is dense in Orlicz-Morrey spaces,
we need to show the well definedness carefully.
Finally, using all of them,
we prove the main results in Section~\ref{sec:proof}.

At the end of this section, we make some conventions. 
Throughout this paper, we always use $C$ to denote a positive constant 
that is independent of the main parameters involved 
but whose value may differ from line to line.
Constants with subscripts, such as $C_p$, are dependent on the subscripts.
If $f\le Cg$, we then write $f\ls g$ or $g\gs f$; 
and if $f \ls g\ls f$, we then write $f\sim g$.

\section{Orlicz and Orlicz-Morrey spaces}\label{sec:Orlicz}

In this section we first recall generalized Young functions.
Next, we recall the definitions of Orlicz and Orlicz-Morrey spaces
with generalized Young functions. 
The Orlicz space is introduced by \cite{Orlicz1932,Orlicz1936}.
For the theory of Orlicz spaces,
see \cite{Kita2009,Kokilashvili-Krbec1991,Krasnoselsky-Rutitsky1961,Maligranda1989,Rao-Ren1991}
for example.
The Orlicz-Morrey spaces investigated 
in \cite{Nakai2004KIT,Nakai2008Studia,Nakai2008KIT}, etc.

For an increasing (i.e. nondecreasing) function 
$\Phi:[0,\infty]\to[0,\infty]$,
let
\begin{equation}\label{aP bP} 
 a(\Phi)=\sup\{t\ge0:\Phi(t)=0\}, \quad 
 b(\Phi)=\inf\{t\ge0:\Phi(t)=\infty\},
\end{equation} 
with convention $\sup\emptyset=0$ and $\inf\emptyset=\infty$.
Then $0\le a(\Phi)\le b(\Phi)\le\infty$.

Let $\biP$ be the set of all increasing functions
$\Phi:[0,\infty]\to[0,\infty]$
such that
\begin{align}\label{ab}
 &0\le a(\Phi)<\infty, \quad 0<b(\Phi)\le\infty, \\
 &\dlim_{t\to+0}\Phi(t)=\Phi(0)=0, \label{lim_0} \\
 &\text{$\Phi$ is left continuous on $[0,b(\Phi))$}, \label{left cont} \\
 &\text{if $b(\Phi)=\infty$, then } 
 \dlim_{t\to\infty}\Phi(t)=\Phi(\infty)=\infty, \label{left cont infty} \\
 &\text{if $b(\Phi)<\infty$, then } 
 \dlim_{t\to b(\Phi)-0}\Phi(t)=\Phi(b(\Phi)) \ (\le\infty). \label{left cont b}
\end{align}

In what follows,
if an increasing and left continuous function $\Phi:[0,\infty)\to[0,\infty)$ satisfies
\eqref{lim_0} and $\dlim_{t\to\infty}\Phi(t)=\infty$,
then we always regard that $\Phi(\infty)=\infty$ and that $\Phi\in\biP$.

For $\Phi\in\biP$,
we recall the generalized inverse of $\Phi$
in the sense of O'Neil \cite[Definition~1.2]{ONeil1965}.

\begin{defn}\label{defn:ginverse}
For $\Phi\in\biP$ and $u\in[0,\infty]$, let
\begin{equation}\label{inverse}
 \Phi^{-1}(u)
 = 
\begin{cases}
 \inf\{t\ge0: \Phi(t)>u\}, & u\in[0,\infty), \\
 \infty, & u=\infty.
\end{cases}
\end{equation}
\end{defn}

Let $\Phi\in\biP$. 
Then $\Phi^{-1}$ is finite, increasing and right continuous on $[0,\infty)$
and positive on $(0,\infty)$.
If $\Phi$ is bijective from $[0,\infty]$ to itself, 
then $\Phi^{-1}$ is the usual inverse function of $\Phi$.
Moreover, if $\Phi\in\biP$, then
\begin{equation}\label{inverse ineq}
 \Phi(\Phi^{-1}(u)) \le u \le  \Phi^{-1}(\Phi(u))
 \quad\text{for all $u\in[0,\infty]$},
\end{equation}
which is a generalization of Property 1.3 in \cite{ONeil1965},
see \cite[Proposition~2.2]{Shi-Arai-Nakai2019Taiwan}.

For $\Phi, \Psi\in\biP$, 
we write $\Phi\approx\Psi$
if there exists a positive constant $C$ such that
\begin{equation*} 
     \Phi(C^{-1}t)\le\Psi(t)\le\Phi(Ct)
     \quad\text{for all}\ t\in[0,\infty].
\end{equation*} 
For functions $P,Q:[0,\infty]\to[0,\infty]$, 
we write $P\sim Q$ 
if there exists a positive constant $C$ such that
\begin{equation*} 
     C^{-1}P(t)\le Q(t)\le CP(t)
     \quad\text{for all}\ t\in[0,\infty].
\end{equation*} 
Then, for $\Phi,\Psi\in\biP$, 
\begin{equation}\label{approx equiv}
 \Phi\approx\Psi \quad \Leftrightarrow \quad \Phi^{-1}\sim\Psi^{-1},
\end{equation}
see \cite[Lemma~2.3]{Shi-Arai-Nakai2019Taiwan}.

Next we recall the definition of the Young function and give its generalization.

\begin{defn}[Young function and its generalization]\label{defn:Young}
A function $\Phi\in\biP$ is called Young function 
(or sometimes also called Orlicz function) 
if 
$\Phi$ is convex on $[0,b(\Phi))$.
Let $\iPy$ be the set of all Young functions.
Let $\biPy$ be the set of all $\Phi\in\biP$ such that
$\Phi\approx\Psi$ for some $\Psi\in \iPy$.
\end{defn}


\begin{defn}\label{defn:D2 n2}
\begin{enumerate}
\item 
A function $\Phi\in\biP$ is said to satisfy the $\Delta_2$-condition,
denoted by $\Phi\in\bdtwo$, 
if there exists a constant $C>0$ such that
\begin{equation}\label{Delta2}
 \Phi(2t)\le C\Phi(t) 
 \quad\text{for all } t>0.
\end{equation}
\item
A function $\Phi\in\biP$ is said to satisfy the $\nabla_2$-condition,
denoted by $\Phi\in\bntwo$, 
if there exists a constant $k>1$ such that
\begin{equation}\label{nabla2}
 \Phi(t)\le\frac1{2k}\Phi(kt) 
 \quad\text{for all } t>0.
\end{equation}
\item
Let $\Delta_2=\iPy\cap\bdtwo$ and $\nabla_2=\iPy\cap\bntwo$.
\end{enumerate}
\end{defn}


Let $(\Omega,\mu)$ be a measure space, 
and let $L^0(\Omega)$ be the set of all measurable functions on $\Omega$.
Then the Orlicz space is defined by the following.

\begin{defn}[Orlicz space]\label{defn:LP}
For $\Phi\in\biPy$, let
\begin{align*}
  \LP(\Omega)
  &= \left\{ f\in L^0(\Omega):
     \int_{\Omega} \Phi(\epsilon |f(x)|)\,d\mu(x)<\infty
               \ \text{for some} \ \epsilon>0 
    \right\}, \\
  \|f\|_{\LP(\Omega)} &=
  \inf\left\{ \lambda>0: 
    \int_{\Omega} \Phi\left(\frac{|f(x)|}{\lambda}\right) d\mu(x)
      \le 1
      \right\}.
\end{align*}
\end{defn}

Then 
$\|\cdot\|_{\LP(\Omega)}$ is a quasi-norm
and thereby $\LP(\Omega)$ is a quasi-Banach space.
If $\Phi\in\iPy$, then $\|\cdot\|_{\LP(\Omega)}$ is a norm 
and thereby $\LP(\Omega)$ is a Banach space.
For $\Phi,\Psi\in\biPy$,
if $\Phi\approx\Psi$, then $\LP(\Omega)=\LPs(\Omega)$
with equivalent quasi-norms.
In the case $\Omega=\R^n$
we always write
$\|\cdot\|_{\LP}$ instead of $\|\cdot\|_{\LP(\R^n)}$,
omitting $(\R^n)$.


For $\Phi\in\biPy$ and 
$\vp:(0,\infty)\to(0,\infty)$,
we define the Orlicz-Morrey space $\LPp(\R^n)$ 
together with $\|\cdot\|_{\LPp}$
by \eqref{OM norm}.
For a ball $B=B(a,r)$,
let $\mu_B = \frac{dx}{|B|\vp(r)}$.
Then we have the following relation:
\begin{equation}\label{LP B norm}
 \|f\|_{\Phi,\vp,B}=\|f\|_{\LP(B,\,\mu_B)}.
\end{equation}
Because of the relation \eqref{LP B norm},
$\|\cdot\|_{\LPp}$ is a quasi-norm, 
and thereby $\LPp(\R^n)$ is a quasi-Banach space.
If $\Phi\in\iPy$, then $\|\cdot\|_{\LPp}$ is a norm 
and thereby $\LPp(\R^n)$ is a Banach space.
If $\Phi\approx\Psi$ and $\vp\sim\psi$, then
$\LPp(\R^n)=\LPsps(\R^n)$ with equivalent quasi-norms.

Next,
we say that a function $\theta:(0,\infty)\to(0,\infty)$ 
satisfies the doubling condition if
there exists a positive constant $C$ such that,
for all $r,s\in(0,\infty)$,
\begin{equation}\label{doubling}
 \frac1C\le\frac{\theta(r)}{\theta(s)}\le C,
 \quad\text{if} \ \ \frac12\le\frac{r}{s}\le2.
\end{equation}
We say that $\theta$ is almost increasing (resp. almost decreasing) if
there exists a positive constant $C$ such that, for all $r,s\in(0,\infty)$,
\begin{equation}\label{almost}
 \theta(r)\le C\theta(s) \quad
 (\text{resp.}\ \theta(s)\le C\theta(r)),
 \quad\text{if $r<s$}.
\end{equation}

In this paper we consider the following class of $\vp:(0,\infty)\to(0,\infty)$.
\begin{defn}\label{defn:cG}
{\rm(i)} Let $\cGdec$ be the set of all functions $\vp:(0,\infty)\to(0,\infty)$
such that 
$\vp$ is almost decreasing
and that
$r\mapsto\vp(r)r^n$ is almost increasing.
That is,
there exists a positive constant $C$ such that, 
for all $r,s\in(0,\infty)$,
\begin{equation*}
 C\vp(r)\ge \vp(s),
 \quad
 \vp(r)r^n\le C\vp(s)s^n, 
 \quad
 \text{if} \ r<s.
\end{equation*}
{\rm(ii)} Let $\cGinc$ be the set of all functions $\vp:(0,\infty)\to(0,\infty)$
such that 
$\vp$ is almost increasing
and that
$r\mapsto\vp(r)/r$ is almost decreasing.
That is,
there exists a positive constant $C$ such that, 
for all $r,s\in(0,\infty)$,
\begin{equation*}
 \vp(r)\le C\vp(s),
 \quad
 C\vp(r)/r\ge \vp(s)/s, 
 \quad
 \text{if} \ r<s.
\end{equation*}
\end{defn}

If $\vp\in\cGdec$ or $\vp\in\cGinc$, then $\vp$ satisfies the doubling condition $\eqref{doubling}$.
Let $\psi:(0,\infty)\to(0,\infty)$.
If $\psi\sim\vp$ for some $\vp\in\cGdec$ (resp. $\cGinc$),
then $\psi\in\cGdec$ (resp. $\cGinc$).

\begin{rem}\label{rem:vp bijective}
Let $\vp\in\cGdec$.
Then there exists $\tvp\in\cGdec$ such that $\vp\sim\tvp$
and that 
$\tvp$ is 
continuous and strictly decreasing, 
see \cite[Proposition~3.4]{Nakai2008Studia}.
Moreover,
if 
\begin{equation}\label{cG* dec}
 \dlim_{r\to0}\vp(r)=\infty,
 \quad
 \dlim_{r\to\infty}\vp(r)=0,
\end{equation}
then $\tvp$ is bijective from $(0,\infty)$ to itself.
\end{rem}

\section{Main results}\label{sec:main}

First we recall the definition of Calder\'on-Zygmund operators 
following \cite{Yabuta1985}.
Let $\Omega$ be the set of all increasing functions $\omega:(0,\infty)\to(0,\infty)$
such that 
$\int_0^1\frac{\omega(t)}{t}dt<\infty$. 

\begin{defn}[{standard kernel}]\label{defn:Kernel}
Let $\omega\in\Omega$. 
A continuous function $K(x,y)$ on $\R^n\times\R^n\setminus\{(x,x)\in\R^{2n}\}$
is said to be a standard kernel of type $\omega$ 
if the following conditions are satisfied:
\begin{gather} 
     |K(x,y)|\le \frac{C}{|x-y|^n} 
     \quad\text{for}\quad x\not=y,
                                                  \label{CZ1} \\     
  \begin{split}
     |K(x,y)-K(x,z)|+|K(y,x)-K(z,x)|
     \le 
     \frac{C}{|x-y|^{n}} 
     \,\omega\!\left(\frac{|y-z|}{|x-y|}\right) & \\
     \text{for}\quad 2|y-z|\le |x-y|. &
  \end{split}
                                                  \label{CZ2} 
\end{gather} 
\end{defn}

\begin{defn}[{Calder\'on-Zygmund operator}]\label{defn:CZO}
Let $\omega\in\Omega$. 
A linear operator $T$ from $\cS(\R^n)$ to $\cS'(\R^n)$
is said to be a Calder\'on-Zygmund operator of type $\omega$,
if $T$ is bounded on $L^2(\R^n)$
and there exists a standard kernel $K$ of type $\omega$ such that,
for $f\in \Cic(\R^n)$,
\begin{equation} 
     Tf(x)=\int_{\R^n} K(x,y)f(y)\,dy, \quad x\notin\supp f.
                                                  \label{CZ3}
\end{equation} 
\end{defn}

\begin{rem}\label{rem:CZO}
If $x\notin\supp f$, then $K(x,y)$ is continuous on $\supp f$ with respect to $y$.
Therefore,
if \eqref{CZ3} holds for $f\in \Cic(\R^n)$,
then \eqref{CZ3} holds for $f\in L^1_{\comp}(\R^n)$.
\end{rem}

It is known by \cite[Theorem 2.4]{Yabuta1985} that 
any Calder\'on-Zygmund operator of type $\omega\in\Omega$
is bounded on $L^p(\R^n)$ for $1<p<\infty$.
This result was extended to Orlicz-Morrey spaces $\LPp(\R^n)$
by \cite{Nakai2008KIT}
as the following:
Assume that $\vp\in\cGdec$ 
and that
there exists a positive constant $C$ such that,
for all $r\in(0,\infty)$,
\begin{equation}\label{int vp}
 \int_r^{\infty}\frac{\vp(t)}{t}\,dt\le C\vp(r).
\end{equation}
Let $\Phi\in\Delta_2\cap\nabla_2$.
For $f\in\LPp(\R^n)$, we define $Tf$ on each ball $B$ by
\begin{equation}\label{Tf def}
 Tf(x)=T(f\chi_{2B})(x)+\int_{\R^n\setminus 2B}K(x,y)f(y)\,dy,
 \quad x\in B.
\end{equation}
Then the first term in the right hand side 
is well defined, 
since $f\chi_{2B}\in\LP(\R^n)$,
and the integral of the second term converges absolutely.
Moreover, 
$Tf(x)$ is independent of the choice of the ball containing $x$.
By this definition we can show that $T$ is a bounded operator on $\LPp(\R^n)$.

For functions $f$ in Orlicz-Morrey spaces, we define $[b,T]f$ on each ball $B$ by
\begin{equation}\label{bTf def}
 [b,T]f(x)=[b,T](f\chi_{2B})(x)+\int_{\R^n\setminus 2B}(b(x)-b(y))K(x,y)f(y)\,dy,
 \quad x\in B,
\end{equation}
see Remark~\ref{rem:well def T} for its well definedness.
Then we have the following theorem.

\begin{thm}\label{thm:comm T}
Let $\Phi,\Psi\in\biPy$, $\vp\in\cGdec$ and $\psi\in\cGinc$.
Let $T$ be a Calder\'on-Zygmund operator of type $\omega\in\Omega$.
\begin{enumerate}
\item 
Let $\Phi,\Psi\in\bdtwo\cap\bntwo$
and $\int_0^1\frac{\omega(t)\log(1/t)}{t}\,dt<\infty$.
Assume that $\vp$ satisfies \eqref{int vp}
and that
there exists a positive constant $C_0$ such that,
for all $r\in(0,\infty)$,
\begin{equation}\label{vp psi T}
 \psi(r)\Phi^{-1}(\vp(r))
 \le C_0\Psi^{-1}(\vp(r)).
\end{equation}
If $b\in\cL_{1,\psi}(\R^n)$,
then
$[b,T]f$ in \eqref{bTf def} is well defined for all $f\in\LPp(\R^n)$ 
and 
there exists a positive constant $C$, independent of $b$ and $f$,
such that
\begin{equation*}
 \|[b,T]f\|_{\LPsp}
 \le
 C\|b\|_{\cL_{1,\psi}}\|f\|_{\LPp}.
\end{equation*}

\item
Conversely, 
assume that 
there exists a positive constant $C_0$ such that,
for all $r\in(0,\infty)$,
\begin{equation}\label{vp psi T inverse}
 C_0\psi(r)\Phi^{-1}(\vp(r))
 \ge \Psi^{-1}(\vp(r)).
\end{equation}
If $T$ is a convolution type such that
\begin{equation}
 Tf(x)=p.v.\int_{\R^n}K(x-y)f(y)\,dy
\end{equation}
with homogeneous kernel $K$ 
satisfying $K(x)=|x|^{-n}K(x/|x|)$, $\int_{S^{n-1}}K=0$, 
$K\in C^{\infty}(S^{n-1})$ and $K\not\equiv0$,
and if $[b,T]$ is bounded from $\LPp(\R^n)$ to $\LPsp(\R^n)$,
then $b$ is in $\cL_{1,\psi}(\R^n)$ and
there exists a positive constant $C$, independent of $b$, 
such that
\begin{equation*}
 \|b\|_{\cL_{1,\psi}}
 \le
 C\|[b,T]\|_{\LPp\to\LPsp},
\end{equation*}
where $\|[b,T]\|_{\LPp\to\LPsp}$ is the operator norm of $[b,T]$
from $\LPp(\R^n)$ to $\LPsp(\R^n)$.
\end{enumerate}
\end{thm}

\begin{rem}\label{rem:comm T}
From the theorem above we have the several corollaries.
\begin{enumerate}
\item 
Take $\Phi(t)=t^p$,
then we have the result for generalized Morrey spaces $L^{(p,\vp)}(\R^n)$.
This case is known by \cite[Theorem~2.1]{Arai-Nakai2018RMC},
which is an extension of Di Fazio and Ragusa~\cite[Theorem~1]{DiFazio-Ragusa1991}.
\item
Take $\vp(r)=1/r^n$,
then we have the result for Orlicz spaces $\LP(\R^n)$.
This case is an extension of Janson~\cite[Theorem]{Janson1978}.
\item
Take $\Phi(t)=\Psi(t)=t^p$, $\vp(r)=1/r^n$ and $\psi\equiv1$, 
then $\LPp(\R^n)=\LPsp(\R^n)=L^p(\R^n)$ and $\cL_{1,\psi}(\R^n)=\BMO(\R^n)$.
This case is the result by Coifman, Rochberg and Weiss~\cite{Coif-Roch-Weiss1976}.
\end{enumerate}
\end{rem}


To state the result on the commutator $[b,\Ir]$
we first mention the boundedness of $\Ir$ on the Orlicz-Morrey spaces.
Let $\Phi,\Psi\in\biPy$ and $\vp\in\cGdec$.
If $\Phi\in\bntwo$ and 
\begin{equation*}
 \int_0^r\frac{\rho(t)}{t}\,dt\;{\Phi}^{-1}(\vp(r)) 
  +\int_r^{\infty}\frac{\rho(t)\,\Phi^{-1}(\vp(t))}{t}\,dt
 \ls
 \Psi^{-1}(\vp(r))
\end{equation*}
holds for all $r\in(0,\infty)$,
then $\Ir$ is bounded from $\LPp(\R^n)$ to $\LPsp(\R^n)$,
see \cite[Theorem~7.3]{Nakai2008Studia}.
More precisely,
in \cite[Theorem~7.3]{Nakai2008Studia}
the author assumed that $\Phi$ and $\Psi$ are bijective,
but it can be extended to $\Phi,\Psi\in\biPy$
by the boundedness of $\Ir$ from $\LP(\R^n)$ to $\LPs(\R^n)$ with $\Phi\in\bntwo$,
which was proven by \cite{Deringoz-Guliyev-Nakai-Sawano-Shi2019Posi}.
See also \cite{Shi-Arai-Nakai2019Taiwan}.

Now we state the result on the commutator $[b,\Ir]$.
For the well definedness of $[b,\Ir]$ on $\LPp(\R^n)$, see Remark~\ref{rem:well def}.

\begin{thm}\label{thm:comm Ir}
Let $\Phi,\Psi\in\biPy$, $\vp\in\cGdec$, $\psi\in\cGinc$
and $\rho:(0,\infty)\to(0,\infty)$.
Assume that $\rho$ satisfies \eqref{int rho} and \eqref{sup rho}.
\begin{enumerate}
\item 
Let $\Phi,\Psi\in\bdtwo\cap\bntwo$. 
Assume that $\vp$ satisfies \eqref{int vp}
and that $r\mapsto\rho(r)/r^{n-\epsilon}$ 
is almost decreasing for some $\epsilon\in(0,n)$.
Assume also that 
there exist positive constants $C_{\rho}$, $C_0$, $C_1$ 
and a function $\Theta\in\bntwo$ such that,
for all $r,s\in(0,\infty)$,
\begin{align} 
\label{rho/rn}
 &C_{\rho}\,\frac{\rho(r)}{r^{n-\epsilon}}\ge \frac{\rho(s)}{s^{n-\epsilon}}, \ \text{if $r<s$},
\\
 \label{rho conti}
 &\left|\frac{\rho(r)}{r^n}-\frac{\rho(s)}{s^n}\right| 
  \le
  C_{\rho}\,|r-s|\frac1{r^{n+1}}\int_0^r\frac{\rho(t)}{t}\,dt,
 \quad \text{if $\frac12\le\frac rs\le2$},
\\
 \label{comm Ir A} 
 &\int_0^r\frac{\rho(t)}{t}\,dt\;{\Phi}^{-1}(\vp(r)) 
  +\int_r^{\infty}\frac{\rho(t)\,\Phi^{-1}(\vp(t))}{t}\,dt
 \le
 C_0\Theta^{-1}(\vp(r)),
\\
 \label{comm Mr A}
 &\psi(r)\Theta^{-1}(\vp(r))
 \le
  C_1 \Psi^{-1}(\vp(r)).
\end{align}
If $b\in\cL_{1,\psi}(\R^n)$,
then
$[b,\Ir]f$ is well defined for all $f\in\LPp(\R^n)$ 
and 
there exists a positive constant $C$, independent of $b$ and $f$,
such that
\begin{equation}\label{comm Ir}
 \|[b,\Ir]f\|_{\LPsp}\le C\|b\|_{\cL_{1,\psi}}\|f\|_{\LPp}.
\end{equation}

\item
Conversely, 
assume that $0<\alpha<n$
and that there exists a positive constant $C_0$ such that,
for all $r\in(0,\infty)$,
\begin{equation*}
 \Psi^{-1}(\vp(r))
 \le 
 C_0r^{\alpha}\psi(r)\Phi^{-1}(\vp(r)).
\end{equation*}
If $[b,\Ia]$ is bounded from $\LPp(\R^n)$ to $\LPsp(\R^n)$,
then $b$ is in $\cL_{1,\psi}(\R^n)$ 
and there exists a positive constant $C$, independent of $b$, such that
\begin{equation}\label{comm Ir c}
 \|b\|_{\cL_{1,\psi}}\le C\|[b,\Ia]\|_{\LPp\to\LPsp},
\end{equation}
where $\|[b,\Ia]\|_{\LPp\to\LPsp}$ is the operator norm of $[b,\Ia]$ from $\LPp(\R^n)$ to $\LPsp(\R^n)$.
\end{enumerate}
\end{thm}

\begin{rem}\label{rem:comm Ir}
From the theorem above we have the several corollaries.
\begin{enumerate}
\item 
Take $\Phi(t)=t^p$,
then we have the result for generalized Morrey spaces $L^{(p,\vp)}(\R^n)$.
This case is known by \cite[Theorem~2.2]{Arai-Nakai2018RMC}.
\item
Take $\vp(r)=1/r^n$,
then we have the result for Orlicz spaces $\LP(\R^n)$.
This case is known by \cite[Theorem~3.13]{Shi-Arai-Nakai2019Taiwan}.
\item
Take $\rho(r)=r^{\alpha}$, $\Phi(t)=t^p$, $\Psi(t)=t^q$, $\vp(r)=1/r^n$ and $\psi\equiv1$, 
then $\LPp(\R^n)=L^p(\R^n)$, $\LPsp(\R^n)=L^q(\R^n)$ and $\cL_{1,\psi}(\R^n)=\BMO(\R^n)$.
This case is the result by Chanillo~\cite{Chanillo1982}.
\end{enumerate}
\end{rem}

For the case $\psi\in\cGdec$, we have the following theorems.

\begin{thm}\label{thm:comm T dec}
Let $\Phi,\Psi\in\bntwo$, $\Phi_0\in\bdtwo$ 
and $\vp, \psi, \theta \in \cGdec$.
Assume that 
\begin{equation}\label{gHolder OM}
 \Phi_0^{-1}(t\psi(r)) \Phi^{-1}(t\vp(r)) \ls \Psi^{-1}(t\theta(r))
\end{equation}
 for all $r,t\in(0,\infty)$.
Assume also that $\vp, \psi, \theta$ satisfy \eqref{int vp}.
Let $T$ be a Calder\'on-Zygmund operator of type $\omega\in\Omega$.
If $b\in\cL^{(\Phi_0,\psi)}(\R^n)$,
then
$[b,T]f$ is well defined for all $f\in \LPp(\R^n)$ 
and 
there exists a positive constant $C$, independent of $b$ and $f$,
such that
\begin{equation*}
 \|[b,T]f\|_{L^{(\Psi,\theta)}}
 \le
 C\|b\|_{\cL^{(\Phi_0,\psi)}}\|f\|_{\LPp}.
\end{equation*}
\end{thm}

\begin{thm}\label{thm:comm Ir dec}
Let $\Phi\in\bntwo$, $\Phi_0\in\bdtwo$, $\Psi\in\biPy$
and $\vp\in\cGdec$.
Assume that $\rho$ satisfies \eqref{int rho} and \eqref{sup rho}
and that $\vp$ satisfies \eqref{int vp}.
Assume also that there exist $\Psi_0\in\bntwo$ and $\Theta\in\biPy$
such that $\Phi^{-1}\Phi_0^{-1}\sim\Psi_0^{-1}$, $\Phi_0^{-1}\Theta^{-1}\ls\Psi^{-1}$
and \eqref{comm Ir A}.
If $b\in\cL^{(\Phi_0,\vp)}(\R^n)$,
then
$[b,\Ir]f$ is well defined for all $f\in \LPp(\R^n)$ 
and there exists a positive constant $C$, independent of $b$ and $f$,
such that
\begin{equation*}
 \|[b,\Ir]f\|_{\LPsp}
 \le
 C\|b\|_{\cL^{(\Phi_0,\vp)}}\|f\|_{\LPp}.
\end{equation*}
\end{thm}

At the end of this section 
we note that, to prove the theorems,
we may assume that $\Phi,\Psi\in\iPy$ instead of  $\Phi,\Psi\in\biPy$.
For example, if $\Phi$ and $\Psi$ satisfy \eqref{vp psi T}
and $\Phi\approx\Phi_1$, $\Psi\approx\Psi_1$, then 
$\Phi_1$ and $\Psi_1$ also satisfy \eqref{vp psi T} by the relation \eqref{approx equiv}.
Moreover,
$\LPp(\R^n)=L^{(\Phi_1,\vp)}(\R^n)$ and
$\LPsp(\R^n)=L^{(\Psi_1,\vp)}(\R^n)$ with equivalent quasi-norms.

\section{Properties on Young functions and Orlicz-Morrey spaces}\label{sec:prop}

Let $a(\Phi)$ and $b(\Phi)$ be as in \eqref{aP bP}.
By the convexity, 
any Young function $\Phi$ is continuous on $[0,b(\Phi))$ and strictly increasing on $[a(\Phi),b(\Phi)]$.
Hence $\Phi$ is bijective from $[a(\Phi),b(\Phi)]$ to $[0,\Phi(b(\Phi))]$.
Moreover, $\Phi$ is absolutely continuous on any closed subinterval in $[0,b(\Phi))$.
That is,
its derivative $\Phi'$ exists a.e. and
\begin{equation}\label{derivative}
 \Phi(t)=\int_0^t\Phi'(s)\,ds, \quad t\in[0,b(\Phi)).
\end{equation}
In this case, 
if $\Phi \in \dtwo$, then $\Phi'$ satisfies the doubling condition, 
see \cite[Lemma~4.4]{Shi-Arai-Nakai2019Taiwan} for example.

\begin{rem}\label{rem:D2 n2}
\begin{enumerate}
\item 
$\bntwo\subset\biPy$ (\cite[Lemma~1.2.3]{Kokilashvili-Krbec1991}).

\item
Let $\Phi\in\biPy$.
Then
$\Phi\in\bdtwo$ if and only if $\Phi\approx\Psi$ for some $\Psi\in\dtwo$,
and,  
$\Phi\in\bntwo$ if and only if $\Phi\approx\Psi$ for some $\Psi\in\ntwo$.

\item
Let $\Phi\in\iPy$.
Then
$\Phi^{-1}$ satisfies the doubling condition by its concavity,
that is, 
\begin{equation}\label{Phi-1 doubl}
 \Phi^{-1}(u)\le\Phi^{-1}(2u)\le2\Phi^{-1}(u)
 \quad\text{for all $u\in[0,\infty]$}.
\end{equation}

\item
Let $\Phi\in\iPy$.
Then $\Phi\in\dtwo$ if and only if 
$t\mapsto\dfrac{\Phi(t)}{t^p}$ is almost decreasing for some $p\in[1,\infty)$.
\end{enumerate}
\end{rem}

Note that, for $\Phi\in\iPy$ and 
for a measurable subset $G\subset\Omega$ with $\mu(G)>0$,
it is known that
\begin{equation}\label{chi Orlicz norm}
 \|\chi_G\|_{\LP(\Omega)}
 =
 \frac1{\Phi^{-1}(1/\mu(G))}.
\end{equation}

\begin{defn}\label{defn:g-inverse}
For a Young function $\Phi$, 
its complementary function is defined by
\begin{equation*}
\cPhi(t)= 
\begin{cases}
   \sup\{tu-\Phi(u):u\in[0,\infty)\}, & t\in[0,\infty), \\
   \infty, & t=\infty.
 \end{cases}
\end{equation*}
\end{defn}

Then $\cPhi$ is also a Young function,
and $(\Phi,\cPhi)$ is called a complementary pair.
For example, 
if $\Phi(t)=t^p/p$, then $\cPhi(t)=t^{p'}/p'$
for $p,p'\in(1,\infty)$ and $1/p+1/p'=1$.
If $\Phi(t)=t$, then
\begin{equation*}
 \cPhi(t)=
\begin{cases}
 0, & t\in[0,1], \\
 \infty, & t\in(1,\infty].
\end{cases}
\end{equation*}

Let $(\Phi,\cPhi)$ be a complementary pair of functions in $\iPy$. 
Then the following inequality holds:
\begin{equation}\label{Phi cPhi r}
 t\le\Phi^{-1}(t) \cPhi^{-1}(t)\le2t
 \quad\text{for}\quad t\in[0,\infty].
\end{equation}
From \eqref{chi Orlicz norm} and \eqref{Phi cPhi r} it follows that
\begin{equation}\label{chi cPhi norm}
 \|\chi_G\|_{\LcP(\Omega)}
 \le
 \mu(G)\Phi^{-1}(1/\mu(G)).
\end{equation}
For the Orlicz spaces we have the following generalized H\"older's inequality;
\begin{equation}\label{g Holder}
 \int_{\Omega} |f(x)g(x)| \,d\mu(x) \le 2\|f\|_{\LP(\Omega)} \|g\|_{\LcP(\Omega)}
 \quad\text{for}\quad f\in\LP(\Omega), \ g\in\LcP(\Omega).
\end{equation}

Let $\Phi\in\iPy$, $\vp:(0,\infty)\to(0,\infty)$ and $B=B(a,r)\subset\R^n$,
and let $\mu_B=dx/(|B|\vp(r))$.
Then 
by the relation \eqref{LP B norm} and \eqref{chi Orlicz norm}
we have
\begin{equation}\label{chi norm B}
 \|\chi_B\|_{\Phi,\vp,B}
 =
 \|\chi_B\|_{\LP(B,\mu_B)}
 =
 \frac1{\Phi^{-1}(1/\mu_B(B))}
 =
 \frac1{\Phi^{-1}(\vp(r))}.
\end{equation}
Moreover, by the relation \eqref{LP B norm} and \eqref{g Holder} we have
\begin{equation}\label{g Holder B}
 \frac1{|B|\vp(r)}\int_{B} |f(x)g(x)| \,dx \le 2\|f\|_{\Phi,\vp,B} \|g\|_{\cPhi,\vp,B}.
\end{equation}

\begin{lem}\label{lem:chi norm 2}
Let $\Phi\in\iPy$ and $\vp \in \cGdec$.
Then
there exists a constant $C\ge1$ such that,
for any ball $B=B(a,r)$, 
\begin{equation}\label{chi norm}
 \frac1{\Phi^{-1}(\vp(r))}
 \le
 \|\chi_B\|_{\LPp}
 \le
 \frac{C}{\Phi^{-1}(\vp(r))}.
\end{equation}
\end{lem}

\begin{proof}
Fix a ball $B=B(a,r)$.
By \eqref{chi norm B} we have
\begin{equation*}
 \frac1{\Phi^{-1}(\vp(r))}
 =
 \|\chi_B\|_{\Phi,\vp,B}
 \le
 \|\chi_B\|_{\LPp}.
\end{equation*}
To show the second inequality in \eqref{chi norm}, let $\lambda=1/{\Phi^{-1}(\vp(r))}$.
Then it is enough to show that, 
for some $C\ge1$ and for all balls $B'=B(b,r')$ with $B\cap B'\ne\emptyset$,
\begin{equation}\label{B'}
 \frac1{\vp(r')}\fint_{B'} \Phi\left(\frac {\chi_B(x)}{C\lambda}\right)dx
 \le1.
\end{equation}
If $B'\subset 3B$,
then $\vp(r') \gs \vp(3r) \sim \vp(r)$. 
Hence
\begin{equation*}
 \frac1{\vp(r')|B'|}\int_{B'} \Phi\left(\frac{\chi_B(x)}{\lambda}\right)dx
 \le
 \frac1{\vp(r')}\Phi\left(\frac{1}{\lambda}\right)
 \ls
 \frac1{\vp(r)}\Phi\left(\frac{1}{\lambda}\right)
 \le 1.
\end{equation*}
In the above we used \eqref{inverse ineq} for the last inequality.
If $B' \cap (3B)^{\complement} \neq \emptyset$ and $B' \cap B \neq \emptyset$,
then $3B' \supset B$.
Hence
$\vp(r')|B'| \sim \vp(3r')|3B'|\gs\vp(r)|B|$
and
\begin{equation*}
 \frac1{\vp(r')|B'|}\int_{B'} \Phi\left(\frac{\chi_B(x)}{\lambda}\right)dx
 \ls
 \frac1{\vp(r)|B|}\int_{B} \Phi\left(\frac1{\lambda}\right)dx
 \le 1.
\end{equation*}
Then, by the convexity of $\Phi$ we have \eqref{B'}.
\end{proof}

\begin{lem}[{\cite[Lemma~4.4]{Shi-Arai-Nakai2019Taiwan}}]\label{lem:eta}
If $\Phi\in\bntwo$, 
then $\Phi((\cdot)^{\theta})\in\bntwo$ for some $\theta\in(0,1)$.
\end{lem}

\begin{lem}\label{lem:p}
Let $\Phi\in\iPy$, $\vp:(0,\infty)\to(0,\infty)$ and $B=B(a,r)\subset\R^n$.
Then
\begin{equation}\label{fint_B f}
 \fint_B |f(x)|\,dx
 \le
 2\Phi^{-1}(\vp(r))\|f\|_{\Phi,\vp,B}.
\end{equation}
Moreover, if $\Phi\in\ntwo$, 
then there exists $p\in(1,\infty)$ such that
\begin{equation*}
 \left(\fint_{B} |f(y)|^p dy\right)^{1/p}
 \le
 C\Phi^{-1}(\vp(r)) \|f\|_{\Phi,\vp,B},
\end{equation*}
where the constant $C$ is independent of $f$ and $B=B(a,r)$.
\end{lem}

\begin{proof}
By \eqref{g Holder B}, \eqref{chi norm B} and \eqref{Phi cPhi r}
we have
\begin{multline*}
 \fint_B |f(x)|\,dx
 \le
 2\vp(r)\|f\|_{\Phi,\vp,B}\|\chi_B\|_{\cPhi,\vp,B} 
 =
 \frac{2\vp(r)}{\cPhi^{-1}(\vp(r))}\|f\|_{\Phi,\vp,B} \\
 \le
 2\Phi^{-1}(\vp(r))\|f\|_{\Phi,\vp,B}.
\end{multline*}
Next we assume that $\Phi\in\ntwo$.
Then by Lemma~\ref{lem:eta} we can take $\theta\in(0,1)$ such that $\Phi((\cdot)^{\theta})\in\bntwo$.
Let $\Phi_{\theta}\in\ntwo$ such that
$\Phi_{\theta}\approx\Phi\left((\cdot)^{\theta}\right)$.
Then ${\Phi_{\theta}}^{-1} \sim (\Phi^{-1})^{1/\theta}$. 
Let $p=1/\theta$.
Then $\||f|^p\|_{\Phi_{\theta},\vp,B}\sim(\|f\|_{\Phi,\vp,B})^p$.
Using $\eqref{fint_B f}$, we have 
\begin{equation*}
 \left(\fint_{B} |f(y)|^p dy\right)^{1/p}
 \le 
 \big(2{\Phi_{\theta}}^{-1}(\vp(r)) \||f|^p\|_{\Phi_{\theta},\vp,B}\big)^{1/p}
 \sim
 \Phi^{-1}(\vp(r)) \|f\|_{\Phi,\vp,B}.
\end{equation*}
\end{proof}

\begin{lem}\label{lem:int Phi vp}
Let $\Phi\in\dtwo$ and $\vp\in\cGdec$.
If $\vp$ satisfies \eqref{int vp},
then
there exists a positive constant $C$ such that,
for all $r\in(0,\infty)$,
\begin{equation}\label{int P vp}
 \int_r^{\infty}\frac{\Phi^{-1}(\vp(t))}{t}\,dt\le C\Phi^{-1}(\vp(r)).
\end{equation}
\end{lem}

\begin{proof}
By Remark~\ref{rem:D2 n2}~(iv) we see that
$t\to\dfrac{\Phi^{-1}(t)}{t^p}$ is almost increasing for some $p\in(0,1]$.
From \eqref{int vp} it follows that
\begin{equation*}
 \int_r^{\infty}\frac{\vp(t)^p}{t}\,dt\le C_p\,\vp(r)^p,
\end{equation*}
for some $C_p>0$, see \cite[Lemma~7.1]{Nakai2008AMS}.
Then
\begin{align*}
 \int_r^{\infty}\frac{\Phi^{-1}(\vp(t))}{t}\,dt
 &=
 \int_r^{\infty}\frac{\Phi^{-1}(\vp(t))}{\vp(t)^p}\frac{\vp(t)^p}{t}\,dt \\
 &\ls
 \frac{\Phi^{-1}(\vp(r))}{\vp(r)^p}\int_r^{\infty}\frac{\vp(t)^p}{t}\,dt 
 \le
 C_p\Phi^{-1}(\vp(r)).
\end{align*}
This shows the conclusion.
\end{proof}

\begin{lem}[{\cite[Thorem~4.1]{Nakai2008Studia}}]\label{lem:gHolder OM}
Let $\Phi_i\in\iPy$ and $\vp_i\in\cGdec$, $i=1,2,3$.
Assume that
\begin{equation*}
 \Phi_1^{-1}(t\vp_1(r))\Phi_3^{-1}(t\vp_3(r))\le C\Phi_2^{-1}(t\vp_2(r))
\end{equation*}
for all $r,t\in(0,\infty)$.
Then
\begin{equation*}
 \|fg\|_{L^{(\Phi_2,\vp_2)}}
 \le
 2C \|f\|_{L^{(\Phi_1,\vp_1)}} \|g\|_{L^{(\Phi_3,\vp_3)}}.
\end{equation*}
\end{lem}

\section{Fractional maximal operator}\label{sec:Mr}

It is well known that
the Hardy-Littlewood maximal operator $M$ is bounded on $L^p(\R^n)$ 
if $1<p\le\infty$.
This boundedness was extended to Orlicz-Morrey spaces by \cite[Theorem~6.1]{Nakai2008Studia}.
Namely, if $\Phi$ is bijective and in $\ntwo$ and $\vp\in\cGdec$, 
then $M$ is bounded on $\LPp(\R^n)$. 
This result 
is valid for any $\Phi\in\bntwo$
by the modular inequality 
\begin{equation*} 
 \int_{\R^n} \Phi(Mf(x))\,dx
 \le
 \int_{\R^n} \Phi(C|f(x)|)\,dx
\end{equation*}
in \cite[Theorem~1.2.1]{Kokilashvili-Krbec1991}.

For the operator $\Mr$ we have the following theorem.
\begin{thm}\label{thm:Mr}
Let $\Phi,\Psi\in\biPy$, $\vp\in\cGdec$
and $\rho:(0,\infty)\to(0,\infty)$.
Assume that 
$\dlim_{r\to\infty}\vp(r)=0$
or that $\Psi^{-1}(t)/\Phi^{-1}(t)$ is almost decreasing on $(0,\infty)$.
If there exists a positive constant $A$ such that,
for all $r\in(0,\infty)$,
\begin{equation}\label{rho Pp Psp}
 \left(\sup_{0<t\le r}\rho(t)\right)\Phi^{-1}(\vp(r))
 \le
 A\Psi^{-1}(\vp(r)),
\end{equation}
then, for any positive constant $C_0$, there exists a positive constant $C_1$ such that, 
for all $f\in \LPp(\R^n)$ with $f\not\equiv0$,
\begin{equation}\label{Mr pointwise}
 \Psi\left(\frac{\Mr f(x)}{C_1\|f\|_{\LPp}}\right) 
 \le
 \Phi\left(\frac{Mf(x)}{C_0\|f\|_{\LPp}}\right),
 \quad x\in\R^n.
\end{equation}
Consequently, 
if $\Phi\in\bntwo$, then 
$\Mr$ is bounded from $\LPp(\R^n)$ to $\LPsp(\R^n)$.
\end{thm}

\begin{rem}\label{rem:Mr}
If $\rho$ is almost increasing or 
if $\Psi^{-1}(t)/\Phi^{-1}(t)$ is almost decreasing, 
then the inequality
$\rho(r)\Phi^{-1}(\vp(r))\ls\Psi^{-1}(\vp(r))$
implies \eqref{rho Pp Psp}.
\end{rem}

\begin{proof}[\bf Proof of Theorem~\ref{thm:Mr}]
We may assume that $\Phi,\Psi\in\iPy$.
We may also assume that $\vp$ is continuous and strictly decreasing, 
see Remark~\ref{rem:vp bijective}.
Let $f\in \LPp(\R^n)$, and let $x\in\R^n$. 
To prove \eqref{Mr pointwise}
we may assume that $\|f\|_{\LPp}=1$ and that $0<Mf(x)<\infty$.

We show that, for any ball $B=B(a,r)$ containing $x$,
\begin{equation}\label{rho mean}
 \rho(r)\fint_B|f|\le C_1 \Psi^{-1}\left(\Phi\left(\frac{Mf(x)}{C_0}\right)\right).
\end{equation}
Then we have the pointwise estimate
\begin{equation*} 
     \Psi\left(\frac{\Mr f(x)}{C_1}\right) 
     \le
     \Phi\left(\frac{Mf(x)}{C_0}\right),
\end{equation*}
which is the conclusion.

To show \eqref{rho mean}, we consider two cases:
\begin{equation*}
 \Phi\left(\frac{Mf(x)}{C_0}\right) \ge \vp(r)
 \quad\text{or}\quad
 \Phi\left(\frac{Mf(x)}{C_0}\right) \le \vp(r).
\end{equation*}
If
$ 
\Phi\left(\frac{Mf(x)}{C_0}\right) \geq \vp(r),
$ 
then by \eqref{fint_B f} and $\|f\|_{\Phi,\vp,B}\le1$,
we have
\begin{equation*}
 \rho(r)\fint_{B}|f|
 \le
 2\rho(r)\Phi^{-1}(\vp(r)).
\end{equation*}
Combining this inequality with \eqref{rho Pp Psp} we have
\begin{equation*}
 \rho(r)\fint_{B}|f|
 \le
 2A\Psi^{-1}(\vp(r)) 
 \le
 2A\Psi^{-1}\left(\Phi\left(\frac{Mf(x)}{C_0}\right)\right).
\end{equation*}
Conversely, let
$ 
 \Phi\left(\frac{Mf(x)}{C_0}\right) \le \vp(r).
$ 
If $\dlim_{r\to\infty}\vp(r)=0$
then we can choose $t_0\in[r,\infty)$ such that
\begin{equation*}
 \Phi\left(\frac{Mf(x)}{C_0}\right) = \vp(t_0).
\end{equation*}
Using \eqref{rho Pp Psp} and \eqref{inverse ineq}, we have
\begin{align*}
 \rho(r)
 \le
 \sup_{0 < t \le t_0}\rho(t)
 \le
 A\frac{\Psi^{-1}(\vp(t_0))}{\Phi^{-1}(\vp(t_0))}
 =
 A\frac{\Psi^{-1}\left(\Phi\left(\frac{Mf(x)}{C_0}\right)\right)}
 {\Phi^{-1}\left(\Phi\left(\frac{Mf(x)}{C_0}\right)\right)} 
 \le
 A\frac{\Psi^{-1}\left(\Phi\left(\frac{Mf(x)}{C_0}\right)\right)}{\frac{Mf(x)}{C_0}}.
\end{align*}
If $\Psi^{-1}(t)/\Phi^{-1}(t)$ is almost decreasing, 
then $\Phi\left(\frac{Mf(x)}{C_0}\right) \le \vp(r)$
implies that
\begin{align*}
 \rho(r)
 \le
 A\frac{\Psi^{-1}(\vp(r))}{\Phi^{-1}(\vp(r))}
 \ls
 \frac{\Psi^{-1}\left(\Phi\left(\frac{Mf(x)}{C_0}\right)\right)}
 {\Phi^{-1}\left(\Phi\left(\frac{Mf(x)}{C_0}\right)\right)} 
 \le
 \frac{\Psi^{-1}\left(\Phi\left(\frac{Mf(x)}{C_0}\right)\right)}{\frac{Mf(x)}{C_0}}.
\end{align*}
In any way we have
\begin{align*}
 \rho(r)\fint_B |f|
 \le
 AC_0\frac{\Psi^{-1}\left(\Phi\left(\frac{Mf(x)}{C_0}\right)\right)}{Mf(x)}\fint_B |f|
 \le
 AC_0\Psi^{-1}\left(\Phi\left(\frac{Mf(x)}{C_0}\right)\right).
\end{align*}
Then we have \eqref{rho mean} and
the proof is complete.
\end{proof}

\section{Orlicz-Campanato spaces and relations to Orlicz-Morrey spaces}\label{sec:OMOC}

In this section we define Orlicz-Campanato spaces 
and investigate their relations to Orlicz-Morrey spaces.

\begin{defn}[Orlicz-Campanato space]\label{defn:OrliczCampanato}
For $\Phi\in\biPy$ and  
$\vp:(0,\infty)\to(0,\infty)$,
let 
\begin{gather*}
 \cLPp(\R^n)
 = \left\{ f\in L^1_{\loc}(\R^n):
   \|f\|_{\cLPp} <\!\infty 
   \right\},            \\
 \|f\|_{\cLPp} 
 = \sup_{B}  \|f-f_B\|_{\Phi,\vp,B},
\end{gather*}
where the supremum is taken over all balls $B$ in $\R^n$
and 
$\|f\|_{\Phi,\vp,B}$ is as in \eqref{PpB norm}.
\end{defn}

Then $\|\cdot\|_{\cLPp}$ is a quasi-norm modulo constant functions
and thereby $\cLPp(\R^n)$ is a quasi-Banach space.
If $\Phi\in\iPy$, then $\|\cdot\|_{\LPp(\R^n)}$ is a norm modulo constant functions
and thereby $\cLPp(\R^n)$ is a Banach space.
If $\Phi\approx\Psi$ and $\vp\sim\psi$, then
$\cLPp(\R^n)=\cLPsps(\R^n)$ with equivalent quasi-norms.

If $\Phi(r)=r^p$ $(1\le p<\infty)$, 
then we denote 
$\cLPp(\R^n)$ by $\cL^{(p,\vp)}(\R^n)$, 
which coincides with $\cL_{p,\vp^p}(\R^n)$ defined by Definition~\ref{defn:gC}.

In this section we prove the following two theorems.
Let $\cC$ be the set of all constant functions.
The first theorem is an extension of \cite[Theorem~2.1]{Mizuhara1995}
and \cite[Theorem~2.1]{Nakai2006Studia}.
\begin{thm}\label{thm:C sim M}
Let $\Phi\in\biPy$ and $\vp\in\cGdec$.
Assume that $\Phi\in\bdtwo$
and that $\vp$ satisfies \eqref{int vp}.
Then
\begin{equation*}
 \cLPp(\R^n)/\cC=\LPp(\R^n)
 \quad\text{and}\quad
 \|f\|_{\cLPp}\sim\|f-\dlim_{r\to\infty}f_{B(0,r)}\|_{\LPp}.
\end{equation*}
More precisely,
for every $f\in\cLPp(\R^n)$, $f_{B(0,r)}$ converges as $r\to\infty$,
and the mapping $f\mapsto f-\dlim_{r\to\infty}f_{B(0,r)}$ is bijective and bicontinuous 
from $\cLPp(\R^n)/\cC$ to $\LPp(\R^n)$.
In this case $\dlim_{r\to\infty}f_{B(a,r)}=\dlim_{r\to\infty}f_{B(0,r)}$ for all $a\in\R^n$.
\end{thm}

\begin{thm}\label{thm:Ms}
Let $\Phi\in\biPy$ and $\vp\in\cGdec$.
If $\Phi\in\bdtwo$,
then there exists a positive constant $C$ such that, for all $f\in L^{1}_{\loc}(\R^n)$,
\begin{equation}\label{Cam<MsMor}
 \|f\|_{\cLPp}\le C\|\Ms f\|_{\LPp}.
\end{equation} 
Moreover,
if $\Phi\in\bntwo$ and $\vp$ satisfies \eqref{int vp},
then 
\begin{equation}\label{Cam sim MsMor}
 C^{-1}\|f\|_{\cLPp}\le \|\Ms f\|_{\LPp}\le C\|f\|_{\cLPp}.
\end{equation}
\end{thm}

By Theorem~\ref{thm:C sim M} and Theorem~\ref{thm:Ms} we have the following corollary.

\begin{cor}\label{cor:Ms}
Let $\Phi\in\biPy$ and $\vp\in\cGdec$.
Assume that $\Phi\in\bdtwo$
and that $\vp$ satisfies \eqref{int vp}.
Then there exist a positive constant $C$ such that, 
for any $f\in\LPp(\R^n)$ satisfying $\dlim_{r\to\infty}f_{B(0,r)}=0$,
\begin{equation}\label{f Msf}
 \|f\|_{\LPp}\le C\|\Ms f\|_{\LPp}.
\end{equation} 
Moreover,
if $\Phi\in\bntwo$,
then
\begin{equation*}
 C^{-1}\|f\|_{\LPp}\le \|\Ms f\|_{\LPp}\le C\|f\|_{\LPp}.
\end{equation*}
\end{cor}

To prove the theorems we prepare
several lemmas.

\begin{lem}\label{lem:fB1-fB2}
Let $\Phi\in\biPy$ and $\vp:(0,\infty)\to(0,\infty)$.
Then,
for any two balls $B_1$ and $B_2$ such that $B_1 \subset B_2$, 
\begin{equation}\label{B1B2}
 |f_{B_1} - f_{B_2}| 
 \le
 2 \frac{|B_2|}{|B_1|} \Phi^{-1}(\vp(r_2)) \|f\|_{\cLPp},
\end{equation}
where $r_2$ is the radius of $B_2$.
\end{lem}

\begin{proof}
By \eqref{fint_B f} we have
\begin{multline*}
 |f_{B_1} - f_{B_2}| 
 \le
 \frac1{|B_1|} \int_{B_1} |f(x) - f_{B_2}|\,dx
 \le 
 \frac{|B_2|}{|B_1|} \fint_{B_2} |f(x)-f_{B_2}|\,dx 
\\
 \le
 2 \frac{|B_2|}{|B_1|} \Phi^{-1}(\vp(r_2)) \|f\|_{\cLPp}.
\qedhere
\end{multline*}
\end{proof}

\begin{lem}\label{lem:fBr-fBs}
Let $\Phi\in\biPy$ and $\vp:(0,\infty)\to(0,\infty)$.
Assume that $\vp$ satisfies the doubling condition.
Then there exists a positive constant $C$ such that,
for any $f\in\cLPp(\R^n)$ and 
for any two balls $B(a,r)$ and $B(b,s)$ satisfying $B(a,r)\subset B(b,s)$,
\begin{equation}\label{fBr-fBs}
 |f_{B(a,r)}-f_{B(b,s)}|
 \le
 C \int_r^{2s}\frac{\Phi^{-1}(\vp(t))}{t}\,dt \,\|f\|_{\cLPp}.
\end{equation}
\end{lem}

\begin{proof}
Let $f\in\cLPp(\R^n)$.
Take balls $B_j=B(a_j,2^jr)$, $j=0,1,2,\dots$, such that
\begin{equation*}
 B(a,r)=B_0\subset B_1\subset B_2 \subset\dots\subset B_{k-1}\subset B(b,s)\subset B_{k}.
\end{equation*}
Then, by \eqref{B1B2} and the doubling condition of $\Phi^{-1}(\vp(\cdot))$ we have
\begin{align*}
 |f_{B(a,r)} - & f_{B(b,s)}|
 \le
 |f_{B_0}-f_{B_1}|+|f_{B_1}-f_{B_2}|+\dots+|f_{B_{k-1}}-f_{B(b,s)}|
\\
 &\le
 2^{n+1} \sum_{j=1}^{k-1} \Phi^{-1}(\vp(2^jr)) \|f\|_{\cLPp}
 + 2\frac{|B(b,s)|}{|B_{k-1}|}\Phi^{-1}(\vp(s)) \|f\|_{\cLPp}
\\
 &\ls
 \sum_{j=1}^{k-1} \int_{2^{j-1}r}^{2^jr} \frac{\Phi^{-1}(\vp(t))}{t}dt \,\|f\|_{\cLPp} 
 + \int_s^{2s} \frac{\Phi^{-1}(\vp(t))}{t}dt \,\|f\|_{\cLPp}
\\
 &\le \int_{r}^{2s} \frac{\Phi^{-1}\left(\vp(t)\right)}{t}dt \,\|f\|_{\cLPp}.
\end{align*}
This shows the conclusion.
\end{proof}

\begin{lem}\label{lem:s(f)}
Let $\Phi\in\biPy$ and $\vp\in\cGdec$.
If $\int_1^{\infty}\frac{\Phi^{-1}(\vp(t))}{t}\,dt<\infty$,
then, for every $f\in\cLPp(\R^n)$,
there exists a constant $\sigma(f)$ such that $\sigma(f)=\dlim_{r\to\infty}f_{B(a,r)}$ 
for all $a\in\R^n$.
\end{lem}

\begin{proof}
Let $f\in\cLPp(\R^n)$. 
By \eqref{fBr-fBs} we see that
\begin{equation*}
 |f_{B(0,r)}-f_{B(0,s)}|
 \le
 C \int_r^{2s}\frac{\Phi^{-1}(\vp(t))}{t}\,dt \,\|f\|_{\cLPp}
 \to0 
 \quad\text{as $r,s\to \infty$ with $r<s$}.
\end{equation*}
Hence $f_{B(0,r)}$ converges as $r$ tends to infinity. 
Let $\sigma(f)=\dlim_{r\to\infty}f_{B(0,r)}$. 
If $|a|\le r$, then $B(a,r)\subset B(0,2r)$.
From \eqref{B1B2} it follows that
\begin{align*}
 |f_{B(a,r)}-\sigma(f)|
 &\le
 |f_{B(a,r)}-f_{B(0,2r)}|+|f_{B(0,2r)}-\sigma(f)|
\\
 &\le
 2^{n+1}\|f\|_{\cLPp}\Phi^{-1}(\vp(2r))+|f_{B(0,2r)}-\sigma(f)|
 \to0 \quad\text{as $r\to\infty$},
\end{align*}
since $\Phi^{-1}(\vp(2r))\to0$ as $r\to\infty$ by the assumption.
\end{proof}

\begin{proof}[\bf Proof of Theorem~\ref{thm:C sim M}]
We may assume that $\Phi\in\dtwo$.
Let $f\in\cLPp(\R^n)$.
Then by the definition of $\cLPp(\R^n)$,
for any ball $B=B(a,r)$,
\begin{equation*}
 \frac1{\vp(r)}
 \fint_{B}\Phi\left(\frac{|f(x)-f_{B}|}{\|f\|_{\cLPp}}\right)dx
 \le 1.
\end{equation*}
Letting $s\to\infty$ in \eqref{fBr-fBs} 
and using Lemma~\ref{lem:s(f)},
we have 
\begin{equation*}
 |f_{B} - \sigma(f)|
 \ls
 \int_{r}^{\infty} \frac{\Phi^{-1}(\vp(t))}{t}dt \,\|f\|_{\cLPp}.
\end{equation*}
By Lemma~\ref{lem:int Phi vp} we have
\begin{equation*}
 |f_{B} - \sigma(f)|
 \le
 C\Phi^{-1}(\vp(r))\,\|f\|_{\cLPp}
\end{equation*}
for some $C\ge1$ independent of $f$.
Then by \eqref{inverse ineq} we have
\begin{equation*}
 \Phi\left(\frac{|f_{B}-\sigma(f)|}{C\|f\|_{\cLPp}}\right)
 \le
 \Phi(\Phi^{-1}(\vp(r)))
 \le
 \vp(r).
\end{equation*}
By the convexity of $\Phi$ we have
\begin{align*}
 &\frac1{\vp(r)}
 \fint_{B}\Phi\left( \frac{|f(x)-\sigma(f)|}{2C \|f\|_{\cLPp}} \right) dx
\\
 &\le
 \frac1{\vp(r)}
 \fint_{B} \frac1{2}\left\{ \Phi\left(\frac{|f(x)-f_{B}|}{C\|f\|_{\cL^{(\Phi,\vp)}}}\right)
  + \Phi\left(\frac{|f_{B}-\sigma(f)|}{C\|f\|_{\cL^{(\Phi,\vp)}}}\right) \right\}\,dx
 \le 1.
\end{align*}
This means that $f-\sigma(f)\in\LPp(\R^n)$ and that
\begin{equation*}
 \|f-\sigma(f)\|_{\LPp} \le 2C\|f\|_{\cLPp}.
\end{equation*}

Conversely, let $f\in\LPp(\R^n)$.
Then by \eqref{fint_B f} we have, for any ball $B=B(a,r)$,
\begin{equation}\label{fB}
 |f_B|
 \le
 \fint_B |f(x)|\,dx 
 \le
  2\Phi^{-1}(\vp(r)) \|f\|_{\LPp}.
\end{equation}
Since $\Phi^{-1}(\vp(r))\to0$ as $r\to\infty$ by the assumption,
we conclude that $\sigma(f)=\dlim_{r\to\infty}f_{B(a,r)}=0$
by Lemma~\ref{lem:s(f)}.
Moreover, from \eqref{fB} and \eqref{chi norm B} it follows that
\begin{align*}
 \|f_B\|_{\Phi,\vp,B}
 &=
 |f_B| \|1\|_{\Phi,\vp,B}
 \le
 2\Phi^{-1}(\vp(r)) \|f\|_{\LPp} \frac1{\Phi^{-1}(\vp(r))}
 = 
 2\|f\|_{\LPp}. 
\end{align*}
Then
\begin{equation*}
 \|f-f_B\|_{\Phi,\vp,B}
 \le
 \|f\|_{\Phi,\vp,B} + \|f_B\|_{\Phi,\vp,B}
 \le
 3\|f\|_{\LPp}.
\end{equation*}
This shows that $f\in\cLPp(\R^n)$ and 
\begin{equation*}
 \|f\|_{\cLPp}
 \le 3 \|f\|_{\LPp}
 = 3 \|f-\sigma(f)\|_{\LPp}.
\end{equation*}
The proof is complete.
\end{proof}

To prove Theorem~\ref{thm:Ms}
we define local versions of the dyadic maximal operator and the dyadic sharp maximal operator.
For any cube $Q\subset\R^n$ centered at $a\in\R^n$ and with side length $2r>0$,
we denote by $\cQd(Q)$ the set of all dyadic cubes with respect to $Q$, that is,
\begin{multline*}
 \cQd(Q) \\
 =
 \left\{Q_{j,k}=a+\prod_{i=1}^{n}[2^{-j}k_ir,2^{-j}(k_i+1)r): j\in\Z,\ k=(k_1,\cdots,k_n)\in\Z^n \right\}.
\end{multline*}
For any cube $Q\subset\R^n$,
let
\begin{align*}
  \Md_{Q} f(x)
  &=
  \sup_{R\in\cQd(Q),\,x\in R\subset Q}\ \fint_R|f(y)|\,dy, \\
  \Msd_{Q} f(x)
  &=
  \sup_{R\in\cQd(Q),\,x\in R\subset Q}\ \fint_R|f(y)-f_R|\,dy.
\end{align*}
Then we have the following lemma.
\begin{lem}\label{lem:Md<Mds}
Let $\Phi\in\Delta_2$ and 
$\Phi(2t)\le C_{\Phi}\Phi(t)$ for all $t\in[0,\infty]$ and some $C_{\Phi}\ge1$.
Then there exists a positive constant $C_{n,\Phi}$ such that, 
for any $f\in L^1_{\loc}(\R^n)$
and any cube $Q$,
\begin{equation}\label{MdQ f<MsdQ f}
 \int_Q \Phi\left( \Md_{Q} f(x) \right) dx
 \le
 C_{n,\Phi} \int_Q \Phi\left( \Msd_Q f(x) \right) dx
 + 2C_{\Phi}\Phi\left(|f|_Q\right)|Q|,
\end{equation}
and 
\begin{equation}\label{f-fQ<MsdQ f}
 \int_Q \Phi\left( \Md_Q(f(x)-f_Q) \right) dx
 \le
 (C_{n,\Phi}+2C_{\Phi}) \int_Q \Phi\left( \Msd_Q f(x) \right) dx.
\end{equation}
\end{lem}

To prove Lemma~\ref{lem:Md<Mds} we use the following local version good $\lambda$ inequality:
\begin{lem}[Tsutsui~\cite{Tsutsui2011}, Komori-Furuya~\cite{Komori2015}]\label{lem:good lambda}
Let $f\in L^1_{\loc}(\R^n)$.
Then, for any cube $Q$, $0<\gamma\le1$ and $\lambda>|f|_Q$, we have
\begin{multline}\label{good lambda}
 \left|\left\{x\in Q:\, \Md_Q f(x)>2\lambda,\Msd_Q f(x)\le\gamma\lambda\right\}\right| \\
 \le
 2^n\gamma 
 \left|\left\{x\in Q:\, \Md_Q f(x)>\lambda\right\}\right|.
\end{multline}
\end{lem}

\begin{proof}[\bf Proof of Lemma~\ref{lem:Md<Mds}]
For $N>0$, let
\begin{equation*}
 I_N
 =
 \int_0^N \Phi^{'}(\lambda) \left| \left\{x\in Q:\Md_Qf(x) > \lambda \right\} \right| d\lambda.
\end{equation*}
If $N>2|f|_Q$, then
\begin{align*}
 I_N 
 &= 
 \int_0^{2|f|_Q} + \int_{2|f|_Q}^N \Phi^{'}(\lambda) \left| \{ x\in Q: \Md_Qf(x) > \lambda\} \right| d\lambda
\\
 &\le
 \Phi\left(2|f|_Q\right)|Q|
 +2\int_{|f|_Q}^{N/2} \Phi^{'}\left(2\lambda\right) \left| \{ x\in Q: \Md_Qf(x) > 2\lambda\} \right| d\lambda.
\end{align*}
By the doubling conditions of $\Phi$ and $\Phi'$
and the good-$\lambda$ inequality \eqref{good lambda}, we have
\begin{align*}
 I_N
 &\le
 2C_{\Phi'}
 \bigg(
  2^n\gamma
  \int_{|f|_Q}^{N/2} 
   \Phi^{'}(\lambda) 
    \left|\{ x\in Q: \Md_Qf(x) > \lambda \}\right| 
  d\lambda 
\\
 &\phantom{*****}
 + \int_{|f|_Q}^{N/2} 
    \Phi^{'}(\lambda) 
     \left| \{ x\in Q: \Msd_Qf(x) > \gamma\lambda \} \right|
   d\lambda
 \bigg)
 + C_{\Phi}\Phi\left(|f|_Q\right)|Q| 
\\
 &\le
 2^{n+1}\gamma C_{\Phi'}I_N 
 + \frac{2C_{\Phi'}}{\gamma} \int_0^{N\gamma/2} \Phi^{'} \left( \frac{\lambda}{\gamma} \right)
   \left| \{ x\in Q: \Msd_Q f(x) > \lambda \} \right| d\lambda
\\
 &\phantom{*********************************}
 + C_{\Phi}\Phi\left(|f|_Q\right)|Q|.
\end{align*}
At this point we pick a $\gamma$ such that $2^{n+1}\gamma C_{\Phi'}=1/2$, then
\begin{align*}
 I_N
 \le
 C_{n,\Phi} 
 \int_0^{\infty} \Phi^{'} \left( \lambda \right) \left| \{ x\in Q: \Msd_Q f(x) > \lambda \} \right| d\lambda
 + 2C_{\Phi}\Phi\left(|f|_Q\right)|Q|. 
\end{align*}
Letting $N\to\infty$, we deduce \eqref{MdQ f<MsdQ f}.
Next, substitute $f-f_Q$ for $f$ in \eqref{MdQ f<MsdQ f}.
Then
\begin{align*}
 &\int_Q \Phi\left( \Md_Q(f(x)-f_Q) \right) dx
\\
 &\le
 C_{n,\Phi} \int_Q \Phi\left( \Msd_{Q} f(x) \right) dx
 + 2C_{\Phi}\Phi\left(|f-f_Q|_Q\right)|Q|
\\
 &\le
 C_{n,\Phi} \int_Q \Phi\left( \Msd_{Q} f(x) \right) dx
 + 2C_{\Phi}\Phi\left(\min_{x\in Q}\Msd_Qf(x)\right)|Q|
\\
 &\le
 (C_{n,\Phi}+2C_{\Phi}) \int_Q \Phi\left( \Msd_{Q} f(x) \right) dx,
\end{align*}
which is \eqref{f-fQ<MsdQ f}.
\end{proof}

\begin{proof}[\bf Proof of Theorem~\ref{thm:Ms}]
To prove \eqref{Cam<MsMor} we may assume that $\|\Ms f\|_{\LPp}=1$.
Then it is enough to prove that
there exists a positive constant $C'$ such that, 
for all balls $B=B(a,r)$, 
\begin{equation}\label{C'}
 \frac1{|B|\vp(r)}\int_B \Phi\left( \frac{ |f(x)-f_B| }{C'} \right) dx
 \le 1.
\end{equation}

Take the cube $Q$ such that $B\subset Q\subset\sqrt{n}B$.
By Jensen's inequality we have
\begin{equation*}
 \Phi\left(|f_Q-f_B|\right) 
 \le
 \Phi\left( \fint_B |f(x)-f_Q |dx \right) 
 \le
 \fint_B \Phi\left( |f(x)-f_Q| \right)dx.
\end{equation*}
Then
\begin{multline}\label{f-fB/2}
 \int_B \Phi\left( \frac{|f(x)-f_B|}{2} \right) dx
 \le
 \frac12\int_B \bigg( \Phi\left(|f(x)-f_Q|\right) + \Phi\left(|f_B-f_Q|\right) \bigg)dx
\\
 \le
 \int_B \Phi\left(|f(x)-f_Q|\right) dx
 \le
 \int_Q \Phi\left(|f(x)-f_Q|\right) dx. 
\end{multline}
By \eqref{f-fQ<MsdQ f} and the fact that $\Msd_Q f\le C_n\Ms f$ for some positive constant $C_n$,
we have
\begin{multline}\label{f-fQ}
 \int_Q \Phi\left(|f(x)-f_Q|\right) dx
 \le
 \int_Q \Phi\left( \Md_Q(f(x)-f_Q) \right) dx
\\
 \le
 (C_{n,\Phi}+2C_{\Phi}) \int_Q \Phi\left( \Msd_Q f(x) \right) dx  
\\
 \le
 (C_{n,\Phi}+2C_{\Phi})\int_{\sqrt{n}B} \Phi\left( C_n\Ms f(x) \right) dx.  
\end{multline}
Take $C_{n,\vp}\ge1$ such that $|\sqrt{n}B|\vp(\sqrt{n}r)\le C_{n,\vp}|B|\vp(r)$.
Then, from \eqref{f-fB/2} and \eqref{f-fQ} it follows that
\begin{equation*}
 \frac1{|B|\vp(r)}\int_B \Phi\left( \frac{|f(x)-f_B|}{2} \right) dx
 \le
 \frac{C_{n,\vp}(C_{n,\Phi}+2C_{\Phi})}{|\sqrt{n}B|\vp(\sqrt{n}r)}
 \int_{\sqrt{n}B} \Phi\left( C_n\Ms f(x) \right) dx,
\end{equation*}
which show that
\begin{multline*}
 \frac1{|B|\vp(r)}\int_B \Phi\left( \frac{|f(x)-f_B|}{2C_{n,\vp}(C_{n,\Phi}+2C_{\Phi})C_n} \right) dx
\\
 \le
 \frac{1}{|\sqrt{n}B|\vp(\sqrt{n}r)}
 \int_{\sqrt{n}B} \Phi\left(\Ms f(x) \right) dx
 \le 1.
\end{multline*}
Therefore we have \eqref{C'}.

Next, we add the assumptions that $\Phi\in\bdtwo$ and that $\vp$ satisfies \eqref{int vp}.
Then
the Hardy-Littlewood maximal operator $M$ is bounded on $\LPp(\R^n)$
and then
\begin{equation}\label{Ms f< f}
 \|\Ms f\|_{\LPp}
 \le
 2 \|Mf\|_{\LPp}
 \le
 C \|f\|_{\LPp}.
\end{equation}
To prove the second inequality in \eqref{Cam sim MsMor}
we may assume that $f\in\cLPp(\R^n)$.
By Theorem~\ref{thm:C sim M}
we see that $f_{B(0,r)}$ converges as $r\to\infty$.
Setting $\sigma(f)=\dlim_{r\to\infty}f_{B(0,r)}$,
we have $\|f-\sigma(f)\|_{\LPp} \le C \|f\|_{\cLPp}$.
Substituting $f-\sigma(f)$ for $f$
in \eqref{Ms f< f},
we have
\begin{align*}
 \|\Ms f\|_{\LPp}
 \le
 C \|f-\sigma(f)\|_{\LPp}
 \le
 C \|f\|_{\cLPp},
\end{align*}
which shows the conclusion.
\end{proof}


\section{Well definedness of the commutators}\label{sec:well}

In this section we prove that
the commutators $[b,T]f$ and $[b,\Ir]f$ is well defined 
for all $b\in\cL_{1,\psi}(\R^n)$ and $f\in\LPp(\R^n)$.

The following theorem is well known.
For the proof, see \cite{Arai-Nakai2018RMC,Nakai2008AMS} for example.

\begin{thm}\label{thm:J-N}
Let $\psi\in\cGinc$. 
Then, for each $p\in(1,\infty)$, 
$\cL_{p,\psi}(\R^n)=\cL_{1,\psi}(\R^n)$
with equivalent norms.
\end{thm}

Using Theorem~\ref{thm:J-N},
we have the following lemma.

\begin{lem}[{\cite[Lemma~4.7]{Arai-Nakai2018RMC}}]\label{lem:int f-fB}
Let $p\in[1,\infty)$ and $\psi\in\cGinc$.
Then there exists a positive constant $C$ dependent only on $n$, $p$ and $\psi$ such that,
for all $f\in\cL_{1,\psi}(\R^n)$ and
for all $x\in\R^n$ and $r,s\in(0,\infty)$,
\begin{equation*}
 \left(\fint_{B(x,s)}|f(y)-f_{B(x,r)}|^p\,dy\right)^{1/p}
 \le
 C\int_r^{s}\frac{\psi(t)}{t}\,dt\,\|f\|_{\cL_{1,\psi}},
 \quad\text{if} \ \ 2r<s. 
\end{equation*}
\end{lem}

\begin{rem}\label{rem:int f-fB}
In Lemma~\ref{lem:int f-fB} we also have
\begin{equation*}
 \left(\fint_{B(x,s)}|f(y)-f_{B(x,r)}|^p\,dy\right)^{1/p}
 \le
 C\left(\log_2\frac sr\right)\psi(s)\,\|f\|_{\cL_{1,\psi}},
 \quad\text{if} \ \ 2r<s, 
\end{equation*}
since 
\begin{equation*}
 \int_{2^jr}^{2^{j+1}r}\frac{\psi(t)}{t}\,dt\ls \psi(s),
\end{equation*}
for $j=0,1,\dots,[\log_2\frac sr]+1$.
\end{rem}


\begin{lem}\label{lem:int2Bc T}
Let $\Phi\in\iPy$ and $\vp\in\cGdec$.
Let $K$ be a standard kernel satisfying \eqref{CZ1}. 
Then there exists a positive constant $C$ such that,
for all $f\in \LPp(\R^n)$ and all balls $B=B(z,r)$,
\begin{equation*}
 \int_{\R^n\setminus2B}|K(x,y)f(y)|\,dy
 \le
 C\int_{2r}^{\infty}\frac{\Phi^{-1}(\vp(t))}{t}\,dt\,\|f\|_{\LPp},
 \quad x\in B.
\end{equation*}
\end{lem}

\begin{proof}
If $x\in B$ and $y\not\in2B$, 
then $|z-y|/2 \le |x-y| \le 3|z-y|/2$.
From \eqref{CZ1} it follows that
$|K(x,y)|\ls|x-y|^{-n}\sim|z-y|^{-n}$.
Then 
\begin{equation*}
 \int_{\R^n\setminus2B}|K(x,y)||f(y)|\,dy 
 \ls
 \int_{\R^n\setminus2B}\frac{|f(y)|}{|z-y|^n}\,dy 
 =
 \sum_{j=1}^{\infty}
 \int_{2^{j+1}B\setminus2^{j}B}\frac{|f(y)|}{|z-y|^n}\,dy.
\end{equation*}
By \eqref{fint_B f}, H\"older's inequality and the doubling condition of $\vp$ we have
\begin{multline*}
 \int_{2^{j+1}B\setminus2^{j}B}\frac{|f(y)|}{|z-y|^n}\,dy 
 \ls
 \fint_{2^{j+1}B}|f(y)|\,dy 
 \ls
 \Phi^{-1} \left( \vp(2^{j+1}r) \right) \|f\|_{\LPp} \\
 \ls
 \int_{2^{j}r}^{2^{j+1}r}\frac{\Phi^{-1}(\vp(t))}{t}\,dt
 \,\|f\|_{\LPp}.
\end{multline*}
Therefore, we have the conclusion.
\end{proof}

\begin{lem}\label{lem:int2Bc psi T}
Let $\Phi\in\ntwo$, $\vp\in\cGdec$, $\psi\in\cGinc$ and $K$ be a standard kernel satisfying \eqref{CZ1}. 
Then there exists a positive constant $C$ such that,
for all $b\in \cL_{1,\psi}(\R^n)$, all $f\in \LPp(\R^n)$ and all balls $B=B(z,r)$,
\begin{multline*}
 \int_{\R^n\setminus2B}|(b(y)-b_{B})K(x,y)f(y)|\,dy \\
 \le
 C\int_{r}^{\infty}\frac{\psi(z,t)}{t}
  \left(\int_{t}^{\infty}
  \frac{\Phi^{-1}(\vp(u))}{u}\,du\right)\!dt
  \,\|b\|_{\cL_{1,\psi}}\|f\|_{\LPp},
 \quad x\in B.
\end{multline*}
\end{lem}

\begin{proof}
If $x\in B$ and $y\not\in2B$, 
then $|z-y|/2\le|x-y|\le 3|z-y|/2$.
From \eqref{CZ1} it follows that
$|K(x-y)|\ls|x-y|^{-n}\sim|z-y|^{-n}$.
Then
\begin{multline*}
 \int_{\R^n\setminus2B}|(b(y)-b_{B})K(x,y)f(y)|\,dy 
 \ls
 \int_{\R^n\setminus2B}\frac{|(b(y)-b_{B})f(y)|}{|z-y|^n}\,dy \\
 =
 \sum_{j=1}^{\infty}
 \int_{2^{j+1}B\setminus2^{j}B}\frac{|(b(y)-b_{B})f(y)|}{|z-y|^n}\,dy.
\end{multline*}
By Lemma~\ref{lem:p} we can find $p\in(1,\infty)$ such that
\begin{equation*}
 \left(\fint_{2^{j+1}B} |f(y)|^p dy\right)^{1/p}
 \ls
 \Phi^{-1}\left(\vp(2^{j+1}r)\right) \|f\|_{\LPp}.
\end{equation*}
By H\"older's inequality, Lemma~\ref{lem:int f-fB} 
and the doubling condition of $\psi$ and $\vp$
we have
\begin{align*}
 &\int_{2^{j+1}B\setminus2^{j}B}\frac{|(b(y)-b_{B})f(y)|}{|z-y|^n}\,dy \\
 &\sim
 \frac{1}{(2^{j+1}r)^n}
  \int_{2^{j+1}B\setminus2^{j}B}|(b(y)-b_{B})f(y)|\,dy \\
 &\ls
 \left(\fint_{2^{j+1}B}|b-b_{B}|^{p'}\,dy\right)^{1/p'} 
 \left(\fint_{2^{j+1}B}|f(y)|^{p}\,dy\right)^{1/p} \\
 &\ls
  \int_{r}^{2^{j+1}r}\frac{\psi(t)}{t}\,dt
  \,\Phi^{-1}(\vp(2^{j+1}r)) 
  \,\|b\|_{\cL_{1,\psi}}\|f\|_{\LPp} \\
 &\ls
 \int_{2^{j}r}^{2^{j+1}r}
  \left(\int_{r}^{u}\frac{\psi(t)}{t}\,dt\right)
  \frac{\Phi^{-1}(\vp(u))}{u}\,du 
  \,\|b\|_{\cL_{1,\psi}}\|f\|_{\LPp}.
\end{align*}
Therefore, 
\begin{multline*}
 \int_{\R^n\setminus2B}|(b(y)-b_{B})K(x,y)f(y)|\,dy \\
 \ls
 \int_{r}^{\infty}
  \left(\int_{r}^{u}\frac{\psi(t)}{t}\,dt\right)
  \frac{\Phi^{-1}(\vp(u))}{u}\,du 
  \,\|b\|_{\cL_{1,\psi}}\|f\|_{\LPp} \\
 =
 \int_{r}^{\infty}\frac{\psi(t)}{t}
  \left(\int_{t}^{\infty}
  \frac{\Phi^{-1}(\vp(u))}{u}\,du\right)\!dt
  \,\|b\|_{\cL_{1,\psi}}\|f\|_{\LPp}.
\end{multline*}
This is the conclusion.
\end{proof}

\begin{rem}\label{rem:well def T}
Under the assumption in Theorem~\ref{thm:comm T} (i), 
let $b\in \cL_{1,\psi}(\R^n)$ and $f\in \LPp(\R^n)$.
Since $\Phi\in\bdtwo$, there exists $p\in(1,\infty)$ such that $t^p \ls \Phi(t)$ for $t \ge 1$,
see Remark~\ref{rem:D2 n2}~(iv).
Then $\LPp(\R^n)\subset\LP_{\loc}(\R^n)\subset L^p_{\loc}(\R^n)$,
which implies $f\in L^p_{\loc}(\R^n)$ 
and $bf \in L^{p_1}_{\loc}(\R^n)$ for all $p_1\in(1,p)$
by Theorem~\ref{thm:J-N}.
Hence, $T(f\chi_{2B})$ and $T(bf\chi_{2B})$ are well defined for any ball $B=B(z,r)$.
By \eqref{int vp}, \eqref{vp psi T} and Lemma~\ref{lem:int Phi vp} we have
\begin{multline}\label{well def T 2}
 \int_{r}^{\infty}\frac{\psi(t)}{t}
  \left(\int_{t}^{\infty}
  \frac{\Phi^{-1}(\vp(u))}{u}\,du\right)\!dt \\
 \ls
 \int_{r}^{\infty}\frac{\psi(t)\Phi^{-1}(\vp(t))}{t}\,dt 
 \ls
 \int_{r}^{\infty}\frac{\Psi^{-1}(\vp(t))}{t}\,dt
 \ls
 \Psi^{-1}(\vp(r)).
\end{multline}
Then, by Lemmas~\ref{lem:int2Bc T} and \ref{lem:int2Bc psi T}, the integrals
\begin{equation*}
 \int_{\R^n\setminus 2B}|K(x,y)f(y)|\,dy
 \quad\text{and}\quad
 \int_{\R^n\setminus 2B}|K(x,y)b(y)f(y)|\,dy
\end{equation*}
converge.
That is, we can write
\begin{equation*}
 [b,T]f(x)
 =
 [b,T](f\chi_{2B})(x)
 +\int_{\R^n\setminus2B}(b(x)-b(y))K(x,y)f(y)\,dy,
 \quad x\in B.
\end{equation*}
Moreover,
if $x\in B_1\cap B_2$,
then, 
taking $B_3$ such that $B_1\cup B_2\subset B_3$,
we have
\begin{align*}
 &\left([b,T](f\chi_{2B_i})(x)
 +\int_{\R^n\setminus2B_i}(b(x)-b(y))K(x,y)f(y)\,dy\right) \\
 &-\left([b,T](f\chi_{2B_3})(x)
 +\int_{\R^n\setminus2B_3}(b(x)-b(y))K(x,y)f(y)\,dy\right) \\
 &=
 -[b,T](f\chi_{2B_3\setminus2B_i})(x)
 +\int_{2B_3\setminus2B_i}(b(x)-b(y))K(x,y)f(y)\,dy=0,
\end{align*}
by \eqref{CZ3}.
That is,
\begin{align*}
 &[b,T](f\chi_{2B_1})(x)
 +\int_{\R^n\setminus2B_1}(b(x)-b(y))K(x,y)f(y)\,dy \\
 &=[b,T](f\chi_{2B_2})(x)
 +\int_{\R^n\setminus2B_2}(b(x)-b(y))K(x,y)f(y)\,dy,
 \quad x\in B_1\cap B_2.
\end{align*}
This shows that $[b,T]f(x)$ in \eqref{bTf def} is independent
of the choice of the ball $B$ containing $x$.
\end{rem}

\begin{lem}\label{lem:int2Bc comm T}
Under the assumption of Theorem~\ref{thm:comm T} (i),
there exists a positive constant $C$ such that,
for all $b\in \cL_{1,\psi}(\R^n)$, all $f\in \LPp(\R^n)$ and all balls $B=B(z,r)$,
\begin{equation*}
 \left|\fint_{B}\left(
  \int_{\R^n\setminus 2B}(b(x)-b(y))K(x,y)f(y)\,dy
  \right)dx\right|
 \le C
 \Psi^{-1}(\vp(B)) \,\|b\|_{\cL_{1,\psi}}\|f\|_{\LPp}.
\end{equation*}
\end{lem}

\begin{proof}
For $x\in B$, let 
\begin{align*}
 G_1(x)&=
 |b(x)-b_{B}|\int_{\R^n\setminus 2B}|K(x,y)f(y)|\,dy, \\
 G_2(x)&=
 \int_{\R^n\setminus 2B}|(b(y)-b_{B})K(x,y)f(y)|\,dy.
\end{align*}
Then
\begin{equation*}
 \left|\int_{\R^n\setminus 2B}(b(x)-b(y))K(x,y)f(y)\,dy\right|
 \le
 G_1(x)+G_2(x).
\end{equation*}
Using Lemmas~\ref{lem:int2Bc T} and \ref{lem:int2Bc psi T},
we have
\begin{equation}\label{for G1}
 \int_{\R^n\setminus 2B}|K(x,y)||f(y)|\,dy 
 \ls
 \int_{2r}^{\infty}\frac{\Phi^{-1}(\vp(t))}{t}\,dt\,\|f\|_{\LPp},
 \quad x\in B,
\end{equation}
and 
\begin{multline}\label{for G2}
 \int_{\R^n\setminus 2B}|b(y)-b_{B}||K(x,y)||f(y)|\,dy \\
 \ls
 \int_{r}^{\infty}\frac{\psi(t)}{t}
  \left(\int_{t}^{\infty}
  \frac{\Phi^{-1}(\vp(u))}{u}\,du\right)\!dt
  \,\|b\|_{\cL_{1,\psi}}\|f\|_{\LPp},
 \quad x\in B.
\end{multline}
Then, using \eqref{for G1}, \eqref{int P vp} and \eqref{vp psi T}, we have
\begin{align*}
 \fint_B G_1(x)\,dx
 &\ls
 \fint_B |b(x)-b_{B}|\,dx\,
 \Phi^{-1}(\vp(r))\|f\|_{\LPp} \\
 &\ls
 \psi(r)\Phi^{-1}(\vp(r))\|b\|_{\cL_{1,\psi}}\|f\|_{\LPp} \\
 &\ls
 \Psi^{-1}(\vp(r)) \|b\|_{\cL_{1,\psi}}\|f\|_{\LPp}.
\end{align*}
Using \eqref{for G2} and \eqref{well def T 2}, we also have
\begin{equation*}
 \fint_B G_2(x)\,dx
 \ls
 \Psi^{-1}(\vp(r))\|b\|_{\cL_{1,\psi}}\|f\|_{\LPp}.
\end{equation*}
Then we have the conclusion.
\end{proof}

\begin{lem}\label{lem:int2Bc}
Let $\Phi \in \iPy$ and $\vp\in\cGdec$.
Assume that $\rho$ satisfies \eqref{int rho} and \eqref{sup rho}. 
Then there exists a positive constant $C$ such that,
for all $f\in \LPp(\R^n)$ and all balls $B(x,r)$,
\begin{equation*}
 \int_{\R^n\setminus B(x,r)}\frac{\rho(|x-y|)}{|x-y|^n}|f(y)|\,dy
 \le
 C\int_{K_1r}^{\infty}\frac{\rho(t)\Phi^{-1}(\vp(t))}{t}\,dt\,\|f\|_{\LPp},
\end{equation*}
where $K_1$ is the constant in \eqref{sup rho}.
\end{lem}

\begin{proof}
Let $B=B(x,r)$. Then 
\begin{equation*}
 \int_{\R^n\setminus B(x,r)}\frac{\rho(|x-y|)}{|x-y|^n}|f(y)|\,dy 
 =
 \sum_{j=0}^{\infty}
 \int_{2^{j+1}B\setminus2^{j}B}\frac{\rho(|x-y|)}{|x-y|^n}|f(y)|\,dy.
\end{equation*}
By \eqref{sup rho}, \eqref{fint_B f}, H\"older's inequality and the doubling condition of $\vp$ we have
\begin{multline*}
 \int_{2^{j+1}B\setminus2^{j}B}\frac{\rho(|x-y|)}{|x-y|^n}|f(y)|\,dy 
 \ls
 \frac{\sup_{2^{j}r\le t\le2^{j+1}r}\rho(t)}{(2^{j+1}r)^n}
  \int_{2^{j+1}B\setminus2^{j}B}|f(y)|\,dy \\
 \ls
 \int_{K_12^{j}r}^{K_22^{j}r}\frac{\rho(t)}{t}\,dt \
 \Phi(\vp(2^{j+1}r)) \|f\|_{\LPp}
 \ls
 \int_{K_12^{j}r}^{K_22^{j}r}\frac{\rho(t)\Phi^{-1}(\vp(t))}{t}\,dt
 \,\|f\|_{\LPp}.
\end{multline*}
Therefore, we have the conclusion.
\end{proof}

\begin{lem}\label{lem:int2Bc psi}
Let $\Phi \in \ntwo$, $\vp\in\cGdec$ and $\psi\in\cGinc$.
Assume that $\rho$ satisfies \eqref{int rho} and \eqref{sup rho}. 
Then there exists a positive constant $C$ such that,
for all $b\in \cL_{1,\psi}(\R^n)$, all $f\in \LPp(\R^n)$ and all balls $B(x,r)$,
\begin{multline*}
 \int_{\R^n\setminus B(x,r)}|b(y)-b_{B(x,r)}|\frac{\rho(|x-y|)}{|x-y|^n}|f(y)|\,dy \\
 \le
 C\int_{K_1r}^{\infty}\frac{\psi(t)}{t}
  \left(\int_{t}^{\infty}
  \frac{\rho(u)\Phi^{-1}(\vp(u))}{u}\,du\right)\!dt
  \,\|b\|_{\cL_{1,\psi}}\|f\|_{\LPp},
\end{multline*}
where $K_1$ is the constant in \eqref{sup rho}.
\end{lem}

\begin{proof}
Let $B=B(x,r)$. Then
\begin{multline*}
 \int_{\R^n\setminus B(x,r)}|b(y)-b_{B}|\frac{\rho(|x-y|)}{|x-y|^n}|f(y)|\,dy 
\\
 =
 \sum_{j=0}^{\infty}
 \int_{2^{j+1}B\setminus2^{j}B}|b(y)-b_{B}|\frac{\rho(|x-y|)}{|x-y|^n}|f(y)|\,dy.
\end{multline*}
By Lemma~\ref{lem:p} we can find $p\in(1,\infty)$ such that
\begin{equation*}
 \left(\fint_{2^{j+1}B} |f(y)|^p dy\right)^{1/p}
 \ls
 \Phi^{-1}\left(\vp(2^{j+1}r)\right) \|f\|_{\LPp}.
\end{equation*}
By \eqref{sup rho}, \eqref{fint_B f}, H\"older's inequality, Lemma~\ref{lem:int f-fB} 
and the doubling condition of $\psi$ and $\vp$
we have
\begin{align*}
 &\int_{2^{j+1}B\setminus2^{j}B}|b(y)-b_{B}|\frac{\rho(|x-y|)}{|x-y|^n}|f(y)|\,dy \\
 &\ls
 \frac{\sup_{2^{j}r\le u\le2^{j+1}r}\rho(u)}{(2^{j+1}r)^n}
  \int_{2^{j+1}B\setminus2^{j}B}|b(y)-b_{B}||f(y)|\,dy \\
 &\ls
 \int_{K_12^{j}r}^{K_22^{j}r}\frac{\rho(u)}{u}\,du
 \left(\fint_{2^{j+1}B}|b-b_{B}|^{p'}\,dy\right)^{1/p'} 
 \left(\fint_{2^{j+1}B}|f(y)|^{p}\,dy\right)^{1/p} \\
 &\ls
 \int_{K_12^{j}r}^{K_22^{j}r}\frac{\rho(u)}{u}\,du
  \int_{r}^{2^{j+1}r}\frac{\psi(t)}{t}\,dt
  \,\Phi^{-1}(\vp(2^{j+1}r)) 
  \,\|b\|_{\cL_{1,\psi}}\|f\|_{\LPp} \\
 &\ls
 \int_{K_12^{j}r}^{K_22^{j}r}
  \left(\int_{K_1r}^{u}\frac{\psi(t)}{t}\,dt\right)
  \frac{\rho(u)\Phi^{-1}(\vp(u))}{u}\,du 
  \,\|b\|_{\cL_{1,\psi}}\|f\|_{\LPp}.
\end{align*}
Therefore, 
\begin{multline*}
 \int_{\R^n\setminus B}|b(y)-b_{B}|\frac{\rho(|x-y|)}{|x-y|^n}|f(y)|\,dy \\
 \ls
 \int_{K_1r}^{\infty}
  \left(\int_{K_1r}^{u}\frac{\psi(t)}{t}\,dt\right)
  \frac{\rho(u)\Phi^{-1}(\vp(u))}{u}\,du 
  \,\|b\|_{\cL_{1,\psi}}\|f\|_{\LPp} \\
 =
 \int_{K_1r}^{\infty}\frac{\psi(t)}{t}
  \left(\int_{t}^{\infty}
  \frac{\rho(u)\Phi^{-1}(\vp(u))}{u}\,du\right)\!dt
  \,\|b\|_{\cL_{1,\psi}}\|f\|_{\LPp}.
\end{multline*}
This is the conclusion.
\end{proof}

\begin{rem}\label{rem:well def}
Under the assumption in Theorem~\ref{thm:comm Ir} (i),
let $b\in \cL_{1,\psi}(\R^n)$ and $f\in \LPp(\R^n)$.
Then $f$ is in $L^p_{\loc}(\R^n)$ and $bf$ is in $L^{p_1}_{\loc}(\R^n)$ for all $p_1<p$
by the same way as in Remark~\ref{rem:well def T}.
Since $\frac{\rho(|y|)}{|y|^n}$ is integrable near the origin with respect to $y$,
$\Ir(|f|\chi_{2B})$ and $\Ir(|bf|\chi_{2B})$ are well defined for any ball $B=B(x,r)$.
By \eqref{comm Ir A} and \eqref{comm Mr A} we have
\begin{equation}\label{well def 1}
 \int_{K_1r}^{\infty}\frac{\rho(t)\Phi^{-1}(\vp(t))}{t}\,dt
 \ls
 \Theta^{-1}(\vp(K_1r))
 \ls
 \Theta^{-1}(\vp(r)),
\end{equation}
and
\begin{multline}\label{well def 2}
 \int_{K_1r}^{\infty}\frac{\psi(t)}{t}
  \left(\int_{t}^{\infty}
  \frac{\rho(u)\Phi^{-1}(\vp(u))}{u}\,du\right)\!dt \\
 \ls
 \int_{K_1r}^{\infty}\frac{\psi(t)\Theta^{-1}(\vp(t))}{t}\,dt 
 \ls
 \int_{K_1r}^{\infty}\frac{\Psi^{-1}(\vp(t))}{t}\,dt
 \ls
 \Psi^{-1}(\vp(r)).
\end{multline}
Then, by Lemmas~\ref{lem:int2Bc} and \ref{lem:int2Bc psi}, the integrals
\begin{equation*}
 \int_{\R^n\setminus 2B}\frac{\rho(|x-y|)}{|x-y|^n}|f(y)|\,dy
 \quad\text{and}\quad
 \int_{\R^n\setminus 2B}\frac{\rho(|x-y|)}{|x-y|^n}|b(y)f(y)|\,dy
\end{equation*}
converge.
That is, the integrals
\begin{equation*}
 \int_{\R^n}\frac{\rho(|x-y|)}{|x-y|^n}f(y)\,dy
 \quad\text{and}\quad
 \int_{\R^n}\frac{\rho(|x-y|)}{|x-y|^n}b(y)f(y)\,dy
\end{equation*}
converge absolutely a.e.~$x$ and we can write
\begin{equation*}
 [b,\Ir]f(x)=
 \int_{\R^n}(b(x)-b(y))\frac{\rho(|x-y|)}{|x-y|^n}f(y)\,dy,
 \quad \text{a.e.}~x.
\end{equation*}
\end{rem}

\begin{lem}\label{lem:int2Bc comm}
Under the assumption of Theorem~\ref{thm:comm Ir} (i),
there exists a positive constant $C$ such that,
for all $b\in \cL_{1,\psi}(\R^n)$, all $f\in \LPp(\R^n)$ and all balls $B=B(z,r)$,
\begin{multline*}
 \left|\fint_{B}\left(
  \int_{\R^n\setminus 2B}(b(x)-b(y))\frac{\rho(|x-y|)}{|x-y|^n}f(y)\,dy
  \right)dx\right|
\\
 \le C
 \Psi^{-1}(\vp(B)) \,\|b\|_{\cL_{1,\psi}}\|f\|_{\LPp}.
\end{multline*}
\end{lem}

\begin{proof}
For $x\in B$, let 
\begin{align*}
 G_1(x)&=
 |b(x)-b_{B}|\int_{\R^n\setminus 2B}\frac{\rho(|x-y|)}{|x-y|^n}|f(y)|\,dy, \\
 G_2(x)&=
 \int_{\R^n\setminus 2B}|b(y)-b_{B}|\frac{\rho(|x-y|)}{|x-y|^n}|f(y)|\,dy.
\end{align*}
Then
\begin{equation*}
 \left|\int_{\R^n\setminus 2B}(b(x)-b(y))\frac{\rho(|x-y|)}{|x-y|^n}f(y)\,dy\right|
 \le
 G_1(x)+G_2(x).
\end{equation*}

Using this estimate and a similar way to Lemmas~\ref{lem:int2Bc} and \ref{lem:int2Bc psi},
we have that, for all $x\in B$,
\begin{align*}
 G_1(x)
 &\ls
 |b(x)-b_{B}|
 \int_{K_1r}^{\infty}
 \frac{\rho(t)\Phi^{-1}\left(\vp(t)\right)}{t}\,dt\, \|f\|_{\LPp}, \\
 G_2(x)
 &\ls
 C\int_{K_1r}^{\infty}\frac{\psi(t)}{t}
  \left(\int_{t}^{\infty}
  \frac{\rho(u)\Phi^{-1}\left(\vp(u)\right)}{u}\,du\right)\!dt
  \,\|b\|_{\cL_{1,\psi}}\|f\|_{\LPp}.
\end{align*}
Then, using \eqref{well def 1} and \eqref{well def 2} also, we have
\begin{align*}
 \fint_B G_1(x)\,dx
 &\ls
 \fint_B |b(x)-b_{B}|\,dx\,
 \Theta^{-1}(\vp(r))\|f\|_{\LPp} \\
 &\ls
 \psi(r)\Theta^{-1}(\vp(r)) \|b\|_{\cL_{1,\psi}}\|f\|_{\LPp} \\
 &\ls
 \Psi^{-1}(\vp(r)) \|b\|_{\cL_{1,\psi}}\|f\|_{\LPp},
\end{align*}
and
\begin{equation*}
 \fint_B G_2(x)\,dx
 \ls
 \Psi^{-1}(\vp(r))\|b\|_{\cL_{1,\psi}}\|f\|_{\LPp}.
\end{equation*}
Then we have the conclusion.
\end{proof}

\section{Proofs of the main results}\label{sec:proof}

We use the following two propositions.
We omit their proofs 
because the proof methods are almost same as \cite[Propositions~5.1 and 5.2]{Arai-Nakai2018RMC} 
and \cite[Proposition~6.2]{Shi-Arai-Nakai2019Taiwan}.

\begin{prop}\label{prop:pointwise CZO}
Let $T$ be a Calder\'on-Zygmund operator of type $\omega$.
Let $\psi\in\cGinc$.
Assume that $\omega$ and $\psi$ satisfy the same assumption in Theorem~\ref{thm:comm T}.
Then, for any $\eta\in(1,\infty)$, there exists a positive constant $C$ such that, 
for all $b\in\cL_{1,\psi}(\R^n)$, $f \in \LPp(\R^n)$ and $x\in\R^n$,
\begin{equation}\label{pointwise}
 M^{\sharp}([b,T]f)(x)
 \le
 C\|b\|_{\cL_{1,\psi}}
 \bigg(
  \big(M_{\psi^{\eta}}(|Tf|^{\eta})(x)\big)^{1/\eta}
  +\big(M_{\psi^{\eta}}(|f|^{\eta})(x)\big)^{1/\eta}
 \bigg),
\end{equation}
where $M_{\psi^{\eta}}$ is the fractional maximal operator defined by
\begin{equation*}
 M_{\psi^{\eta}}f(x)=\sup_{B(a,r)\ni x} \psi(r)^{\eta}\fint_{B(a,r)}|f(y)|\,dy,
 \quad x\in\R^n.
\end{equation*}
\end{prop}

\begin{prop}\label{prop:pointwise Ir}
Assume that $\rho:(0,\infty)\to(0,\infty)$ satisfies \eqref{int rho}.
Let
\begin{equation}\label{rho*}
 \rho^*(r)=\int_0^r\frac{\rho(t)}{t}\,dt.
\end{equation}
Assume that 
the condition \eqref{rho conti} holds
and that
$r\mapsto\rho(r)/r^{n-\epsilon}$ is almost decreasing for some $\epsilon>0$. Assume also that
\begin{equation}\label{int converge}
 \int_{r}^{\infty}\frac{\rho(t)\Phi^{-1}(\vp(t))}{t}\,dt<\infty,
 \quad
 \int_{r}^{\infty}\frac{\psi(t)}{t}
  \left(\int_{t}^{\infty}
   \frac{\rho(u)\Phi^{-1}(\vp(u))}{u}\,du\right)\!dt<\infty,
\end{equation}
Then, for any $\eta\in(1,\infty)$, there exists a positive constant $C$ such that, 
for all $b\in\cL_{1,\psi}(\R^n)$, $f \in \LPp(\R^n)$ and $x\in\R^n$,
\begin{equation}\label{pointwise Ir}
 \Ms([b,\Ir]f)(x)
 \le
 C\|b\|_{\cL_{1,\psi}}
 \bigg(
  \big(M_{\psi^\eta}(|\Ir f|^{\eta})(x)\big)^{1/\eta}
  +\big(M_{(\rho^*\psi)^{\eta}}(|f|^{\eta})(x)\big)^{1/\eta}
 \bigg)
\end{equation}
where $M_{(\rho^*\psi)^{\eta}}$ is the fractional maximal operator defined by
\begin{equation*}
 M_{(\rho^*\psi)^{\eta}}f(x)=\sup_{B(a,r)\ni x} (\rho^*(r)\psi(r))^{\eta}\fint_{B(a,r)}|f(y)|\,dy,
 \quad x\in\R^n.
\end{equation*}
\end{prop}

We note that the condition \eqref{int converge} is used to prove the well definedness of $[b,\Ir]f$.

Next, we note that,
for $\theta\in(0,\infty)$,
\begin{equation}\label{theta}
 \||g|^{\theta}\|_{\LPp}=\left(\|g\|_{L^{(\Phi((\cdot)^{\theta}), \vp))}}\right)^{\theta}.
\end{equation}

\begin{proof}[\bf Proof of Theorem~\ref{thm:comm T} (i)]
First note that $T$ is bounded on $\LPp(\R^n)$
as we state just before Theorem~\ref{thm:comm T}.
We can take $\eta\in(1,\infty)$ such that $\Phi((\cdot)^{1/\eta})\in\bntwo$
by Lemma~\ref{lem:eta}.
Then, from \eqref{vp psi T} it follows that
\begin{equation*}
 \psi(r)^{\eta}\Phi^{-1}(\vp(r))^{\eta}
 \le
 {C_0}^{\eta}\Psi^{-1}(\vp(r))^{\eta}.
\end{equation*}
By Theorem~\ref{thm:Mr} with this condition
we have the boundedness of 
$M_{\psi^{\eta}}$ from $L^{(\Phi((\cdot)^{1/\eta}),\vp)}(\R^n)$ to $L^{(\Psi((\cdot)^{1/\eta}),\vp)}(\R^n)$.
Using this boundedness and \eqref{theta}, we have
\begin{align*}
 &\left\|\big(M_{\psi^\eta}(|Tf|^{\eta})\big)^{1/\eta}\right\|_{\LPsp}
 =
 \left(\left\|M_{\psi^\eta}(|Tf|^{\eta})\right\|_{L^{(\Psi((\cdot)^{1/\eta}),\vp)}}\right)^{1/\eta} \\
 &\ls
 \left(\left\||Tf|^{\eta}\right\|_{L^{(\Phi((\cdot)^{1/\eta}),\vp)}}\right)^{1/\eta} 
 =
 \|Tf\|_{\LPp} 
 \ls
 \|f\|_{\LPp}.
\end{align*}
and 
\begin{align*}
 \|(M_{\psi^{\eta}}(|f|^{\eta}))^{1/\eta}\|_{\LPsp}
 &=
 \big(\|M_{\psi^{\eta}}(|f|^{\eta})\|_{L^{(\Psi((\cdot)^{1/\eta}),\vp)}}\big)^{1/\eta} \\
 &\ls
 \big(\||f|^{\eta}\|_{L^{(\Phi((\cdot)^{1/\eta}),\vp)}}\big)^{1/\eta} 
 =
 \|f\|_{\LPp}.
\end{align*}
Then, using Proposition~\ref{prop:pointwise CZO}, we have
\begin{equation}\label{Ms bf T}
 \|\Ms([b,T]f)\|_{\LPsp}
 \ls
 \|b\|_{\cL_{1,\psi}}\|f\|_{\LPp}.
\end{equation}
Therefore, 
if we show that, for $B_r=B(0,r)$,
\begin{equation}\label{mean vanish T}
 \fint_{B_r}[b,T]f
 \to0
 \quad\text{as}\quad r\to\infty,
\end{equation}
then by Corollary~\ref{cor:Ms} we have 
\begin{equation}\label{comm ineq T}
 \|[b,T]f\|_{\LPsp}
 \ls
 \|b\|_{\cL_{1,\psi}}\|f\|_{\LPp},
\end{equation}
which is the conclusion.

In the following we show \eqref{mean vanish T}.

\noindent
{\bf Case 1}: First we show 
\eqref{mean vanish T} for all $f\in \LPp(\R^n)$ with compact support.
Let $\supp f\subset B_s=B(0,s)$ with $s\ge1$.
Then $f\in L^p(\R^n)$ and $bf\in L^{p_1}(\R^n)$ for some $1<p_1<p<\infty$
(see Remark~\ref{rem:well def T}).
Since $T$ is bounded on Lebesgue spaces,
we see that both $(bTf)\chi_{B_{2s}}$ and $T(bf)\chi_{B_{2s}}$ are in $L^{1}(\R^n)$
and that 
\begin{equation*}
 \fint_{B_r}(bTf)\chi_{B_{2s}}\to0,
 \quad
 \fint_{B_r}T(bf)\chi_{B_{2s}}\to0
 \quad\text{ as $r\to\infty$}.
\end{equation*}
If $x\not\in B_{2s}$ and $y\in B(0,s)$, 
then $|x|/2\le|x-y|\le3|x|/2$.
By \eqref{CZ1} and \eqref{CZ3} we have
\begin{equation}\label{T sup}
 |Tf(x)|
 \ls
 \frac{1}{|x|^n}\|f\|_{L^1},
 \quad
 |T(bf)(x)|
 \ls
 \frac{1}{|x|^n}\|bf\|_{L^1},
 \quad
 x\not\in B_{2s},
\end{equation}
which yields
\begin{equation*}
 b_{B_{2s}}
 \fint_{B_r}(Tf)(1-\chi_{B_{2s}})\to0,
 \quad
 \fint_{B_r}(T(bf))(1-\chi_{B_{2s}})\to0
 \quad\text{ as $r\to\infty$}.
\end{equation*}
Next, we show
\begin{equation}\label{int bTf 0}
 \fint_{B_r}(b-b_{B_{2s}})(Tf)(1-\chi_{B_{2s}})\to0
 \quad\text{ as $r\to\infty$}.
\end{equation}
Then we have \eqref{mean vanish T} 
for all $f\in \LPp(\R^n)$ with compact support.

Now, since $\Psi \in \dtwo$, there exists $p \in (1,\infty)$ such that $\Psi^{-1}(u) \ls u^{1/p} \, (u \le 1)$. Let $\nu = \frac{2p}{2p-1}$, then
\begin{multline*}
 \left|\fint_{B_r}(b-b_{B_{2s}})(Tf)(1-\chi_{B_{2s}})\right|
\\
 \le 
 \left(\fint_{B_r}|b-b_{B_{2s}}|^{\nu'}\right)^{1/\nu'}
 \left(\fint_{B_r} \left|(Tf)(1-\chi_{B_{2s}})\right|^{\nu}\right)^{1/\nu}.
\end{multline*}
From Lemma~\ref{lem:int f-fB}, Remark~\ref{rem:int f-fB} and \eqref{vp psi T} it follows that
\begin{multline}\label{b nu T}
 \left(\fint_{B_r}|b-b_{B_{2s}}|^{\nu'}\right)^{1/\nu'}
 \ls
 \int_{2s}^{r}\frac{\psi(t)}{t}\,dt\,\|b\|_{\cL_{1,\psi}} \\
 \ls
 \psi(r)\log r\,\|b\|_{\cL_{1,\psi}} 
 \ls
 \frac{\Psi^{-1}(\vp(r))}{\Phi^{-1}(\vp(r))}\log r\,\|b\|_{\cL_{1,\psi}}.
\end{multline}
From \eqref{T sup} it follows that
\begin{equation}\label{f nu T}
 \left(\int_{B_r\setminus B_{2s}} |Tf(x)|^{\nu}\,dx \right)^{1/\nu} 
 \ls
 \left(\int_{B_r\setminus B_{2s}}
  \left( \frac{1}{|x|^n}\|f\|_{L^1} \right)^{\nu} \,dx
 \right)^{1/\nu}
 \ls
 \|f\|_{L^1}.
\end{equation}
By \eqref{b nu T} and \eqref{f nu T}
we have
\begin{align*}
 &\left|\fint_{B_r}(b-b_{B_{2s}})(Tf)(1-\chi_{B_{2s}})\right| \\
 &\ls
 \frac{\Psi^{-1}(\vp(r))}{\Phi^{-1}(\vp(r))} \log r\,\|b\|_{\cL_{1,\psi}}
 \frac1{r^{n/\nu}}\,\|f\|_{L^1} 
  =
  \frac{\log r}{r^{n/\nu}}
 \frac{\Psi^{-1}(\vp(r))}{\Phi^{-1}(\vp(r))}
 \,\|b\|_{\cL_{1,\psi}}
 \|f\|_{L^1} \\
 &\ls
 \frac{\log r}{r^{n/\nu}} \frac{\vp(r)^{1/p}}{\vp(r)} \|b\|_{\cL_{1,\psi}} \|f\|_{L^1} 
 =
 \frac{\log r}{r^{\frac{n}{2p}} (r^n\vp(r))^{1-\frac1{p}}} \|b\|_{\cL_{1,\psi}} \|f\|_{L^1} \\
 &\to0
 \quad\text{ as $r\to\infty$},
\end{align*}

Therefore, we have \eqref{mean vanish T} 
and \eqref{comm ineq T} for all $f\in \LPp(\R^n)$ with compact support.

\noindent
{\bf Case 2}: For general $f\in \LPp(\R^n)$, using Case 1, 
we have 
\begin{equation*}
 \|[b,T](f\chi_{B_{2r}})\|_{\LPsp}
 \ls
 \|b\|_{\cL_{1,\psi}}\|f\chi_{B_{2r}}\|_{\LPp}
 \le
 \|b\|_{\cL_{1,\psi}}\|f\|_{\LPp}.
\end{equation*}
Then, by \eqref{fint_B f},
\begin{multline*}
 \fint_{B_r}[b,T](f\chi_{B_{2r}})
 \le
 \Psi^{-1}(\vp(r))\|[b,T](f\chi_{B_{2r}})\|_{\LPsp} 
\\
 \ls
 \Psi^{-1}(\vp(r))
 \|b\|_{\cL_{1,\psi}}\|f\|_{\LPp}.
\end{multline*}
Combining this with Lemma~\ref{lem:int2Bc comm T}, we have
\begin{equation*}
 \fint_{B_r}[b,T]f
 \ls
 \Psi^{-1}(\vp(r)) \|b\|_{\cL_{1,\psi}}\|f\|_{\LPp},
\end{equation*}
which implies \eqref{mean vanish T}.
Therefore, we have \eqref{comm ineq T} for all $f\in \LPp(\R^n)$.
The proof is complete.
\end{proof}

\begin{proof}[\bf Proof of Theorem~\ref{thm:comm T} (ii)]
We use the method by Janson~\cite{Janson1978}.
Since $1/K(z)$ is many times infinitely differentiable in an open set,
we may choose $z_0\ne0$ and $\delta>0$ such that 
$1/K(z)$ can be expressed in the neighborhood $|z-z_0|<2\delta$
as an absolutely convergent Fourier series, $1/K(z)=\sum a_je^{iv_j\cdot z}$.
(The exact form of the vectors $v_j$ is irrelevant.
For example, 
if the cube centered at $z_0$ of side length $4\delta$ is contained in the open set, 
then we can take $v_j=2\pi j/(4\delta)$, $j\in\Z^n$.)

Set $z_1=z_0/\delta$. If $|z-z_1|<2$, we have the expansion
\begin{equation*}
 \frac1{K(z)}=\frac{\delta^{-n}}{K(\delta z)}=\delta^{-n}\sum a_je^{iv_j\cdot\delta z}.
\end{equation*}
Choose now any ball $B=B(x_0,r)$.
Set $y_0=x_0-rz_1$ and $B'=B(y_0,r)$.
Then, if $x\in B$ and $y\in B'$,
\begin{equation*}
 \left|\frac{x-y}r-z_1\right|\le\left|\frac{x-x_0}r\right|+\left|\frac{y-y_0}r\right|\le2.
\end{equation*}
Denote $\sgn(f(x)-f_{B'})$ by $s(x)$. 
Then
\begin{align*}
 &\int_B|b(x)-b_{B'}|\,dx
 =
 \int_B(b(x)-b_{B'})s(x)\,dx \\
 &=
 \frac1{|B'|}\int_B\int_{B'}(b(x)-b(y))s(x)\,dy\,dx \\
 &=
 \frac1{|B'|}\int_{\R^n}\int_{\R^n}
  (b(x)-b(y))\frac{r^nK(x-y)}{K(\frac{x-y}r)}s(x)\chi_{B}(x)\chi_{B'}(y)\,dy\,dx \\
 &=
 \frac{r^n\delta^{-n}}{|B'|}\int_{\R^n}\int_{\R^n}
  (b(x)-b(y))K(x-y) \sum a_je^{iv_j\cdot\delta\frac{x-y}r}
   s(x)\chi_{B}(x)\chi_{B'}(y)\,dy\,dx.
\end{align*}
Here, we set $C=\delta^{-n}|B(0,1)|^{-1}$ and
\begin{equation*}
 g_j(y)=e^{-iv_j\cdot\delta\frac{y}r}\chi_{B'}(y),
 \quad
 h_j(x)=e^{iv_j\cdot\delta\frac{x}r}s(x)\chi_{B}(x).
\end{equation*}
Then
{\allowdisplaybreaks 
\begin{align*}
 &\int_B|b(x)-b_{B'}|\,dx \\
 &=
 C\sum a_j 
 \int_{\R^n}\int_{\R^n}
 (b(x)-b(y))K(x-y) g_j(y)h_j(x) \,dy\,dx \\
 &=
 C\sum a_j \int_{\R^n} ([b,T]g_j)(x)h_j(x)\,dx \\
 &\le
 C\sum |a_j| \int_{\R^n} |([b,T]g_j)(x)||h_j(x)|\,dx \\
 &=
 C\sum |a_j| \int_{B} |([b,T]g_j)(x)|\,dx \\
 &\le
 C\sum |a_j| |B|\Psi^{-1}(\vp(r)) \|[b,T]g_j\|_{\LPsp} \\
 &\le
 C\|[b,T]\|_{\LPp \to \LPsp} |B|\Psi^{-1}(\vp(r))
 \sum |a_j| \|g_j\|_{\LPp}.
\end{align*}
}
By Lemma~\ref{lem:chi norm 2}
we have that
$\|g_j\|_{\LPp} = \|\chi_{B'}\|_{\LPp} \sim \frac1{\Phi^{-1}(\vp(B'))}$.
Then
\begin{equation*}
 \int_B|b(x)-b_{B'}|\,dx 
 \ls
 \|[b,T]\|_{\LPp \to \LPsp}
 |B| \frac{\Psi^{-1}(\vp(B))}{\Phi^{-1}(\vp(B))}.
\end{equation*}
By \eqref{vp psi T inverse} we have
\begin{equation*}
 \frac1{\psi(B)}\fint_B|b(x)-b_{B}|\,dx 
 \le \frac2{\psi(B)}\fint_B|b(x)-b_{B'}|\,dx 
 \ls\|[b,T]\|_{\LPp \to \LPsp}.
\end{equation*}
That is, $\|b\|_{\cL_{1,\psi}}\ls\|[b,T]\|_{\LPp \to \LPsp}$
and we have the conclusion.
\end{proof}

\begin{proof}[\bf Proof of Theorem~\ref{thm:comm Ir} {\rm (i)}]
We may assume that $\Phi,\Psi\in\dtwo\cap\ntwo$ and $\Theta\in\ntwo$.
We can choose $\eta\in(1,\infty)$ such that 
$\Phi((\cdot)^{1/\eta})$, $\Psi((\cdot)^{1/\eta})$ and $\Theta((\cdot)^{1/\eta})$ are in $\ntwo$
by Lemma~\ref{lem:eta}.
Then from \eqref{comm Mr A} it follows that 
\begin{equation*}
 \psi(r)^{\eta}\Theta^{-1}(\vp(r))^{\eta}
 \le
 C_1^{\eta} \Psi^{-1}(\vp(r))^{\eta}.
\end{equation*}
Hence, 
by Theorem~\ref{thm:Mr} 
we see that 
$M_{\psi^\eta}$ is bounded from $L^{(\Theta((\cdot)^{1/\eta}),\vp)}(\R^n)$ 
to $L^{(\Psi((\cdot)^{1/\eta}),\vp)}(\R^n)$.
Moreover, as we mentioned just before Theorem~\ref{thm:comm Ir}
$\Ir$ is bounded from $\LPp(\R^n)$ to $\LTp(\R^n)$ by \eqref{comm Ir A}.
Then, using \eqref{theta}, we have
\begin{multline*}
 \left\|\big(M_{\psi^\eta}(|\Ir f|^{\eta})\big)^{1/\eta}\right\|_{\LPsp}
 =
 \left(\left\|M_{\psi^\eta}(|\Ir f|^{\eta})\right\|_{L^{(\Psi((\cdot)^{1/\eta}),\vp)}}\right)^{1/\eta} \\
 \ls
 \left(\left\||\Ir f|^{\eta}\right\|_{L^{(\Theta((\cdot)^{1/\eta}),\vp)}}\right)^{1/\eta} 
 =
 \|\Ir f\|_{\LTp} 
 \ls
 \|f\|_{\LPp}.
\end{multline*}
From \eqref{comm Ir A} and \eqref{comm Mr A} 
it follows that
\begin{equation*}
 (\rho^*(r)\psi(r))^{\eta}\left({\Phi}^{-1}(\vp(r))\right)^{\eta}
 \le
 (C_0C_1)^{\eta} \left(\Psi^{-1}(\vp(r))\right)^{\eta}. 
\end{equation*}
By using Theorem~\ref{thm:Mr},
we have the boundedness of $M_{(\rho^*\psi)^{\eta}}$ 
from $L^{(\Phi((\cdot)^{1/\eta}),\vp)}$ to $L^{(\Psi((\cdot)^{1/\eta}),\vp)}$.
That is,
\begin{align*}
 \left\|\big(M_{(\rho^*\psi)^{\eta}}(|f|^{\eta})\big)^{1/\eta}\right\|_{\LPsp}
 &=
 \left(\left\|M_{(\rho^*\psi)^{\eta}}(|f|^{\eta})\right\|_{L^{(\Psi((\cdot)^{1/\eta}),\vp)}}\right)^{1/\eta} \\
 &\ls
 \left(\left\|| f|^{\eta}\right\|_{L^{(\Phi((\cdot)^{1/\eta}),\vp)}}\right)^{1/\eta} 
 =
 \|f\|_{\LPp}.
\end{align*}
Therefore, 
if we show that, for $B_r=B(0,r)$,
\begin{equation}\label{mean vanish}
 \fint_{B_r}[b,\Ir]f
 \to0
 \quad\text{as}\quad r\to\infty,
\end{equation}
then we have 
\begin{equation}\label{comm ineq}
 \|[b,\Ir]f\|_{\LPsp}
 \ls
 \|b\|_{\cL_{1,\psi}}\|f\|_{\LPp},
\end{equation}
by Corollary~\ref{cor:Ms}.

In the following we show \eqref{mean vanish}.

\noindent
{\bf Case 1}: First we show 
\eqref{mean vanish} for all $f\in \LPp(\R^n)$ with compact support.
Let $\supp f\subset B_s=B(0,s)$ with $s\ge1$.
Then $f\in L^p(\R^n)$ and $bf\in L^{p_1}(\R^n)$ for some $1<p_1<p<\infty$
(see Remark~\ref{rem:well def}).
Since $\frac{\rho(|y|)}{|y|^n}$ is locally integrable with respect to $y$,
we see that $(b\Ir f)\chi_{B_{2s}}$ and $\Ir(bf)\chi_{B_{2s}}$ are in $L^{1}(\R^n)$
and that 
\begin{equation*}
 \fint_{B_r}(b\Ir f)\chi_{B_{2s}}\to0,
 \quad
 \fint_{B_r}\Ir(bf)\chi_{B_{2s}}\to0
 \quad\text{ as $r\to\infty$}.
\end{equation*}
If $x\not\in B_{2s}$ and $y\in B(0,s)$, 
then $|y|<|x-y|$ and $|x|/2\le|x-y|\le3|x|/2$,
\begin{equation}\label{rho sup rho}
 \rho(|x-y|)
 \le
 \sup_{|x|/2\le t\le3|x|/2}\rho(t).
\end{equation}
Then we have
\begin{equation*}
 \frac{\rho(|x-y|)}{|x-y|^n}
 \ls
 \frac{\sup_{|x|/2\le t\le3|x|/2}\rho(t)}{|x|^n}
 \sim
 \sup_{|x|/2\le t\le3|x|/2}\frac{\rho(t)}{t^n},
\end{equation*}
and
\begin{equation*}\label{Irf sup}
 |\Ir f(x)|
 \ls
 \sup_{|x|/2\le t\le3|x|/2}\frac{\rho(t)}{t^n}
 \|f\|_{L^1},
 \quad
 |\Ir(bf)(x)|
 \ls
 \sup_{|x|/2\le t\le3|x|/2}\frac{\rho(t)}{t^n}
 \|bf\|_{L^1}.
\end{equation*}
From almost decreasingness of $t\mapsto\rho(t)/t^{n-\epsilon}$
for some $\epsilon\in(0,n)$,
it follows that $\frac{\rho(t)}{t^n}\to0$ as $t\to\infty$, 
which yields
\begin{equation*}
 b_{B_{2s}}
 \fint_{B_r}(\Ir f)(1-\chi_{B_{2s}})\to0,
 \quad
 \fint_{B_r}(\Ir(bf))(1-\chi_{B_{2s}})\to0
 \quad\text{ as $r\to\infty$}.
\end{equation*}
Next, we show
\begin{equation}\label{int bIrf 0}
 \fint_{B_r}(b-b_{B_{2s}})(\Ir f)(1-\chi_{B_{2s}})\to0
 \quad\text{ as $r\to\infty$}.
\end{equation}
Then we have \eqref{mean vanish} 
for all $f\in \LPp(\R^n)$ with compact support.

Now, since $\Psi \in \dtwo$, 
there exists $p \in (1,\infty)$ such that $\Psi^{-1}(u) \ls u^{1/p}$ $(u \le 1)$. 
Let $\nu = \frac{2p}{2p-1}$, then
\begin{multline*}
 \left|\fint_{B_r}(b-b_{B_{2s}})(\Ir f)(1-\chi_{B_{2s}})\right|
\\
 \le
 \left(\fint_{B_r}|b-b_{B_{2s}}|^{\nu'}\right)^{1/\nu'}
 \left(\fint_{B_r} \left|(\Ir f)(1-\chi_{B_{2s}})\right|^{\nu}\right)^{1/\nu}.
\end{multline*}
From Lemma~\ref{lem:int f-fB}, Remark~\ref{rem:int f-fB} and \eqref{comm Mr A} it follows that
\begin{multline}\label{b nu}
 \left(\fint_{B_r}|b-b_{B_{2s}}|^{\nu'}\right)^{1/\nu'}
 \ls
 \int_{2s}^{r}\frac{\psi(t)}{t}\,dt\,\|b\|_{\cL_{1,\psi}} \\
 \ls
 \psi(r)\log r\,\|b\|_{\cL_{1,\psi}} 
 \ls
 \frac{\Psi^{-1}(\vp(r))}{\Theta^{-1}(\vp(r))}\log r
 \,\|b\|_{\cL_{1,\psi}}.
\end{multline}
For $j=0,1,2,\dots$, from \eqref{rho sup rho} and \eqref{sup rho} it follows that
\begin{align*}
 &\left(\int_{2^{j+2}B_s\setminus2^{j+1}B_s} |\Ir f(x)|^{\nu}\,dx \right)^{1/\nu} \\
 &\ls
 \left(\int_{2^{j+2}B_s\setminus2^{j+1}B_s} \left( \frac{\sup_{|x|/2\le t\le3|x|/2}\rho(t)}{|x|^n}
 \|f\|_{L^1} \right)^{\nu} \,dx \right)^{1/\nu}\\
 &\ls
 (2^js)^{(-n\nu+n)/\nu}\sup_{2^js\le t\le 3\cdot2^{j+1}s}\rho(t) \,\|f\|_{L^1} 
 \ls
 \int_{2^jK_1s}^{3\cdot2^jK_2s}\frac{\rho(t)}{t}\,dt\,\|f\|_{L^1},
\end{align*}
since $s\ge1$.
Take the integer $j_0$ such that $r\le2^{j_0+2}s<2r$.
Then, by \eqref{comm Ir A},
\begin{multline}\label{f nu}
 \left(\fint_{B_r} \left|(\Ir f)(1-\chi_{B_{2s}})\right|^{\nu} \right)^{1/\nu}
 \le
 \frac1{r^{n/\nu}}\sum_{j=0}^{j_0}
  \left(\int_{2^{j+2}B_s\setminus2^{j+1}B_s} |\Ir f|^{\nu} \right)^{1/\nu} \\
 \ls
 \frac1{r^{n/\nu}}\int_0^{3K_2r/2}\frac{\rho(t)}{t}\,dt\,\|f\|_{L^1} 
 \ls
 \frac1{r^{n/\nu}} \frac{\Theta^{-1}(\vp(r))}{\Phi^{-1}(\vp(r))}
 \,\|f\|_{L^1}.
\end{multline}
By \eqref{b nu} and \eqref{f nu}, we have
{\allowdisplaybreaks 
\begin{align*}
 &\left|\fint_{B_r}(b-b_{B_{2s}})(\Ir f)(1-\chi_{B_{2s}})\right| 
\\
 &\ls
 \frac{\Psi^{-1}(\vp(r))}{\Theta^{-1}(\vp(r))}\log r
 \frac1{r^{n/\nu}}\frac{\Theta^{-1}(\vp(r))}{\Phi^{-1}(\vp(r))}
 \,\|b\|_{\cL_{1,\psi}}
 \|f\|_{L^1} 
\\
 &=
 \frac{\log r}{r^{n/\nu}}
 \frac{\Psi^{-1}(\vp(r))}{\Phi^{-1}(\vp(r))}
 \,\|b\|_{\cL_{1,\psi}}
 \|f\|_{L^1} 
\\
 &\ls
 \frac{\log r}{r^{n/\nu}} \frac{\vp(r)^{1/p}}{\vp(r)} \|b\|_{\cL_{1,\psi}} \|f\|_{L^1} 
 =
 \frac{\log r}{r^{\frac{n}{2p}} (r^n\vp(r))^{1-\frac1{p}}} \|b\|_{\cL_{1,\psi}} \|f\|_{L^1} 
\\
 &\to0
 \quad\text{ as $r\to\infty$}.
\end{align*}
}
Therefore, we have \eqref{mean vanish} 
and \eqref{comm ineq} for all $f\in \LPp(\R^n)$ with compact support.

\noindent
{\bf Case 2}: For general $f\in \LPp (\R^n)$, using Case 1, 
we have 
\begin{equation*}
 \|[b,\Ir](f\chi_{B_{2r}})\|_{\LPsp}
 \ls
 \|b\|_{\cL_{1,\psi}}\|f\chi_{B_{2r}}\|_{\LPp}
 \le
 \|b\|_{\cL_{1,\psi}}\|f\|_{\LPp}.
\end{equation*}
Then, by \eqref{fint_B f},
\begin{multline*}
 \fint_{B_r}[b,\Ir](f\chi_{B_{2r}})
 \le
 \Psi^{-1}(\vp(r))\|[b,\Ir](f\chi_{B_{2r}})\|_{\LPsp}
\\
 \ls
 \Psi^{-1}(\vp(r))
 \|b\|_{\cL_{1,\psi}}\|f\|_{\LPp}.
\end{multline*}
Combining this with Lemma~\ref{lem:int2Bc comm}, we have
\begin{equation*}
 \fint_{B_r}[b,\Ir]f
 \ls
 \Psi^{-1}(\vp(r)) \|b\|_{\cL_{1,\psi}}\|f\|_{\LPp},
\end{equation*}
which implies \eqref{mean vanish}.
Therefore, we have \eqref{comm ineq} for all $f\in \LPp(\R^n)$.
The proof is complete.
\end{proof}

\begin{proof}[\bf Proof of Theorem~\ref{thm:comm Ir} {\rm (ii)}]
In a similar way to the proof of Theorem~\ref{thm:comm T} (ii), 
we can conclude that
$\|b\|_{\cL_{1,\psi}}\ls\|[b,\Ir]\|_{\LPp\to\LPsp}$,
by calculating $|z|^{n-\alpha}$ instead of $1/K(z)$.
\end{proof}


\begin{proof}[\bf Proof of Theorem~\ref{thm:comm T dec}]
Let $B_r=B(0,r)$.
By Theorem~\ref{thm:C sim M} we have that,
for every $b\in\cL^{(\Phi_0,\psi)}(\R^n)$, $b_{B_r}$ converges as $r\to\infty$
and $\|b-\dlim_{r\to\infty}b_{B_r}\|_{L^{(\Phi_0,\psi)}}\sim\|b\|_{\cL^{(\Phi_0,\psi)}}$.
Let $b_0=b-\dlim_{r\to\infty}b_{B_r}$.
Then $\|b_0\|_{L^{(\Phi_0,\psi)}} \sim \|b\|_{\cL^{(\Phi_0,\psi)}}$ and $[b,T]f=b_0Tf-T(b_0f)$.
Using the boundedness of $T$ on $\LPp(\R^n)$ and on $L^{(\Psi,\theta)}(\R^n)$
and generalized H\"older's inequality (Lemma~\ref{lem:gHolder OM})
with the assumption \eqref{gHolder OM},
we have 
\begin{align*}
 \|[b,T]f\|_{L^{(\Psi, \theta)}}
 &\le
 \|b_0Tf\|_{L^{(\Psi, \theta)}} + \|T(b_0f)\|_{L^{(\Psi, \theta)}} \\
 &\ls
 \|b_0\|_{L^{(\Phi_0, \psi)}}\|Tf\|_{\LPp} + \|b_0f\|_{L^{(\Psi, \theta)}} \\
 &\ls
 \|b_0\|_{L^{(\Phi_0, \psi)}} \|f\|_{\LPp}
 \sim \|b\|_{\cL^{(\Phi_0,\vp)}}\|f\|_{\LPp}.
\end{align*}
This is the conclusion.
\end{proof}

\begin{proof}[\bf Proof of Theorem~\ref{thm:comm Ir dec}]
We use the same method as the proof of Theorem~\ref{thm:comm T dec}.
For $b\in\cL^{(\Phi_0,\vp)}(\R^n)$, let $b_0=b-\dlim_{r\to\infty}b_{B_r}$.
Then $\|b_0\|_{L^{(\Phi_0,\vp)}}\sim\|b\|_{\cL^{(\Phi_0,\vp)}}$ and $[b,\Ir]f=b_0\Ir f-\Ir(b_0f)$.
As we mentioned just before Theorem~\ref{thm:comm Ir}
$\Ir$ is bounded from $\LPp(\R^n)$ to $L^{(\Theta,\vp)}(\R^n)$
by the assumption \eqref{comm Ir A}.
Moreover, we see that 
$\Ir$ is bounded from $L^{(\Psi_0,\vp)}(\R^n)$ to $L^{(\Psi,\vp)}(\R^n)$,
since
\begin{multline*}
 \int_0^r\frac{\rho(t)}{t}\,dt\,\Psi_0^{-1}(\vp(r))
 +\int_r^{\infty}\frac{\rho(t) \Psi_0^{-1}(\vp(t))}{t}\,dt 
\\
 \sim
 \int_0^r\frac{\rho(t)}{t}\,dt\,\Phi^{-1}(\vp(r))\Phi_0^{-1}(\vp(r))
 +\int_r^{\infty}\frac{\rho(t) \Phi^{-1}(\vp(t))\Phi_0^{-1}(\vp(t))}{t}\,dt
\\
 \ls
 \left(\int_0^r\frac{\rho(t)}{t}\,dt\,\Phi^{-1}(\vp(r))
 +\int_r^{\infty}\frac{\rho(t) \Phi^{-1}(\vp(t))}{t}\,dt \right) \Phi_0^{-1}(\vp(r))
\\
 \ls
 \Theta^{-1}(\vp(r)) \Phi_0^{-1}(\vp(r))
 \ls
 \Psi^{-1}(\vp(r)). 
\end{multline*}
In the above we use the almost decreasingness of $r\mapsto\Phi_0^{-1}(\vp(r))$.
Then, using these boundedness of $\Ir$ 
and generalized H\"older's inequality (Lemma~\ref{lem:gHolder OM}),
we have 
\begin{align*}
 \|[b,\Ir]f\|_{\LPsp}
 &\le
 \|b_0\Ir f\|_{\LPsp} + \|\Ir(b_0f)\|_{\LPsp} \\
 &\ls
 \|b_0\|_{L^{(\Phi_0,\vp)}}\|\Ir f\|_{L^{(\Theta,\vp)}}+\|b_0f\|_{L^{(\Psi_0,\vp)}} \\
 &\ls
 \|b_0\|_{L^{(\Phi_0,\vp)}}\|f\|_{\LPp}
 \sim \|b\|_{\cL^{(\Phi_0,\vp)}}\|f\|_{\LPp}.
\end{align*}
This is the conclusion.
\end{proof}



\bigskip

\begin{flushright}
\begin{minipage}{70mm}
\noindent
Minglei Shi \\
Department of Mathematics \\
Ibaraki University \\
Mito, Ibaraki 310-8512, Japan \\
18nd206l@vc.ibaraki.ac.jp \\
stfoursml@gmail.com
\end{minipage}
\\[3ex]
\begin{minipage}{70mm}
\noindent
Ryutaro Arai \\
Department of Mathematics \\
Ibaraki University \\
Mito, Ibaraki 310-8512, Japan \\
ryutaro.arai.math@vc.ibaraki.ac.jp \\
araryu314159@gmail.com
\end{minipage}
\\[3ex]
\begin{minipage}{70mm}
\noindent
Eiichi Nakai \\
Department of Mathematics \\
Ibaraki University \\
Mito, Ibaraki 310-8512, Japan \\
eiichi.nakai.math@vc.ibaraki.ac.jp  
\end{minipage}
\end{flushright}

\end{document}